\documentclass[11pt]{amsart}

\usepackage{amsmath}
\usepackage{amssymb}
\usepackage{amsthm}
\usepackage{mathrsfs}
\usepackage{amsfonts}
\usepackage{amscd}
\usepackage{graphicx,amssymb}
\usepackage[colorlinks=true,linkcolor=skirv_blue,citecolor=skirv_blue]{hyperref}
\usepackage{color}
\definecolor{skirv_blue}{RGB}{0,0,205}

\input Diagrams.tex

\usepackage{fullpage}
\usepackage{comment}

\theoremstyle{plain}
\newtheorem{thm}{Theorem}[section]
\newtheorem*{thm*}{Theorem}
\newtheorem{cor}[thm]{Corollary}
\newtheorem*{cor*}{Corollary}
\newtheorem{prop}[thm]{Proposition}
\newtheorem{lem}[thm]{Lemma}

\theoremstyle{definition}
\newtheorem{dfn}[thm]{Definition}
\newtheorem{rem}[thm]{Remark}
\newtheorem{exmp}[thm]{Example}

\numberwithin{equation}{section}

\renewcommand{\O}{\mathcal{O}}
\renewcommand{\L}{\mathcal{L}}
\newcommand{\N}{\mathcal{N}}
\newcommand{\C}{\mathbb{C}}
\newcommand{\Q}{\mathbb{Q}}
\newcommand{\g}{\mathfrak{g}}
\DeclareMathOperator{\im}{im}

\DeclareMathOperator{\irr}{irr}

\DeclareMathOperator{\df}{def}

\DeclareMathOperator{\tr}{tr}

\DeclareMathOperator{\rk}{rk}

\DeclareMathOperator{\Spec}{Spec}

\DeclareMathOperator{\Pic}{Pic}

\DeclareMathOperator{\Bun}{Bun}

\DeclareMathOperator{\Hom}{Hom}

\DeclareMathOperator{\End}{End}

\DeclareMathOperator{\Aut}{Aut}

\DeclareMathOperator{\Ext}{Ext}

\DeclareMathOperator{\Tor}{Tor}

\title{A global analogue of the Springer resolution for $SL_2$}
\author{Michael Skirvin}
\begin{document}
\begin{abstract}
The global nilpotent cone $\N$ is a singular stack associated to the choice of an algebraic group $G$, a smooth projective curve $X$, and a line bundle $\L$ on $X$, which is of fundamental importance to the Geometric Langlands Program, and which is of emerging importance to
the Classical Langlands Program. In analogy with the ordinary
Springer resolution, we construct and study a resolution of
singularities of $\N$ in the special case where $G=SL_2$. As an
immediate application, we prove that $\N$ is equidimensional and
also provide an enumeration of its irreducible components. We hope
this is the first step in constructing a global Springer resolution
for an arbitrary reductive group.
\end{abstract}

\topmargin -0.5 truein \maketitle

\section{Introduction}
Fix a triple $(X,G,\L)$, in which $X$ denotes a smooth projective
curve, $G$ denotes a reductive algebraic group, and $\L$ indicates a
line bundle on $X$. Given this data, we may consider the Hitchin
moduli stack $\mathcal{M}$, as well as the Hitchin fibration
\begin{equation*}
\chi^{\mathrm{Hit}}:\mathcal{M} \to \mathcal{A},
\end{equation*}
where $\mathcal{A}$ denotes a vector space. Beginning with \cite{Hi}, the Hitchin fibration has received much attention in recent years, playing a significant role in the Geometric Langlands Program (\cite{BD}, \cite{F}), the Classical Langlands Program (\cite{Ngo}), and non-Abelian Hodge theory (\cite{Sim}, \cite{CHM}).

While there are various points of view on the Hitchin fibration, we
prefer to think of it as a global analogue of the adjoint quotient map associated to the Lie algebra $\g$ of $G$. Then, just as the
ordinary nilpotent cone $\N \mathrm{ilp}$ is defined as the fiber
over zero of the adjoint quotient map, we may define the global
nilpotent cone $\N$ to be the fiber over zero of the Hitchin
fibration. While $\N \mathrm{ilp}$ has played a significant role in
geometric representation theory for over $30$ years (see Section \ref{ss:ord spring}), $\N$ has only received a modest, but
important, study in the Geometric Langlands Program, and has
received almost no study whatsoever in the Classical Langlands
Program and in non-Abelian Hodge theory.\footnote{It is common to
study the restriction of the Hitchin fibration to various subspaces
of $\mathcal{A}$ which, in particular, do not contain zero.} This
relative lack of attention can largely be explained by the fact that $\N$ is the most difficult Hitchin fiber to understand, in part
because it is the most singular.

Thus, the purpose of this paper is to begin a program which aims to
understand the global nilpotent cone as thoroughly as the ordinary
nilpotent cone. To this end, we construct an explicit resolution of
singularities of $\N$ when $G=SL_2$, which we think of as a global
analogue of the Springer resolution $\widetilde{\N \mathrm{ilp}}$.

More specifically, we identify the Drinfeld/Laumon compactification
$\overline{\Bun}_B(X)$ (see Section \ref{ss:B compactification}) as
the appropriate global analogue of the flag variety $G/B$, and then
define the partial global Springer resolution $\widehat{\N}$ (see
Definition \ref{global springer}) to be a particular closed substack
of
\begin{equation*}
\N \underset{\Bun_G(X)}{\times} \overline{\Bun}_B(X).
\end{equation*}

\begin{thm*}[see Sections \ref{ss:genus 0} and \ref{ss:genus 1}]
If $X$ is a rational or elliptic curve, then $\widehat{\N}$ is a
resolution of singularities of $\N$.
\end{thm*}

When the genus of $X$ is greater than $1$, $\widehat{\N}$ is not a
resolution of singularities of $\N$ because it is not smooth. In
order to understand its singularities, we instead study a simpler stack, which is denoted $C\mathcal{G}(X)$, by producing a smooth map
\begin{equation*}
\widehat{\N} \to C\mathcal{G}(X).
\end{equation*}
The stack $C\mathcal{G}(X)$ is easier to understand than $\widehat{\N}$ because it is simply the moduli of pairs $(\lambda,s)$, where $\lambda$ is a line bundle on $X$ such that $h^0(X,\lambda) \geq 1$, and $s \in H^0(X,\lambda)$. Thus, the geometry of $\widehat{\N}$ is closely tied to the classical geometry of line bundles and divisors on curves, as studied in \cite{ACGH}.

We are then able to define the global Springer resolution
$\widetilde{\N}$, which is birationally equivalent to
$\widehat{\N}$, by first resolving the singularities of
$C\mathcal{G}(X)$.

\begin{thm*}[see Section \ref{ss:higher genus}]
If the genus of $X$ is at least $2$, then $\widetilde{\N}$ is a
resolution of singularities of $\N$.
\end{thm*}

As an immediate consequence, we obtain a proof that $\N$ is
equidimensional. While this fact is known when $\L$ is the canonical bundle $\omega_X$ (see, for example, \cite{Gi}), it does not seem to have appeared in print for general $\L$. Furthermore, we are able to give an enumeration of the irreducible components of $\N$

\begin{cor*}[see Sections \ref{ss:genus 0}-\ref{ss:higher genus}]
Suppose that $\deg(\L) \geq 2g$. Then $\N$ is equidimensional of dimension $\deg(\L)+g-1$. Its irreducible components are indexed by:
\begin{enumerate}
\item All integers $d > -\frac{1}{2}\deg(\L)$.
\item The square roots of $\L^{-1}$, of which there are $2^{2g}$.
\end{enumerate}
\end{cor*}

In Section \ref{ss:lower degree}, a similar analysis is carried out for the the stable part of $\N$, which has only finitely many irreducible components.

\subsection{The ordinary Springer resolution}\label{ss:ord spring} Let $G$ be a complex reductive
group with Lie algebra $\g$. Then the adjoint action of $G$ on $\g$
induces the \em adjoint quotient map\em
\begin{equation*}
\chi:\g \to \mathfrak{c},
\end{equation*}
where $\mathfrak{c}$ denotes the adjoint quotient $\g/\!/G$. The
notation $\chi$ is chosen because when $G = GL_n$, the adjoint
quotient map is simply the map which associates to a matrix $A$ the
(non-leading coefficients of the) characteristic polynomial of $A$.
Generalizing the fact that the matrix $A$ is nilpotent if and only
if its characteristic polynomial is $t^n$, we define the \em
nilpotent cone \em of $\g$ to be
\begin{equation*}
\N \mathrm{ilp} := \chi^{-1}(0).
\end{equation*}
$\N \mathrm{ilp}$ is a singular algebraic variety which is normal and which possesses an explicit resolution of singularities known as the \em Springer resolution \em (\cite{S}). To describe this resolution, let $G/B$ denote the flag variety of $G$, which may be described as the variety of Borel subalgebras of $\g$. Letting
\begin{equation*}
\widetilde{\N \mathrm{ilp}} := \{ (x,\mathfrak{b}) \in \N
\mathrm{ilp} \times G/B: x \in \mathfrak{b} \},
\end{equation*}
the Springer resolution is given by the projection map
\begin{equation*}
\mu:\widetilde{\N \mathrm{ilp}} \to \N \mathrm{ilp}.
\end{equation*}
Furthermore, there is an isomorphism $\widetilde{\N \mathrm{ilp}}
\simeq T^{\ast}(G/B)$ between the total space of the Springer
resolution and the cotangent bundle of the flag variety. Letting
$\widetilde{\g}$ denote the variety of pairs $(x,\mathfrak{b}) \in
\g \times G/B$ such that $x \in \mathfrak{b}$, we have thus far
described the following spaces and maps:
\begin{equation}\label{classical springer diagram}
\begin{diagram}
G/B & \lTo & \widetilde{\N \mathrm{ilp}} \simeq T^{\ast}(G/B) &
\rInto &
\widetilde{\g}\\
&& \dTo_{\mu} && \dTo_p \\
&& \N \mathrm{ilp} & \rInto & \g \\
&& \dTo_{\chi} && \dTo_{\chi} \\
&& \{ 0 \} & \rInto & \mathfrak{c} \\
\end{diagram}
\end{equation}

Although this paper is only concerned with geometry, much of the
motivation for the Springer resolution comes from representation
theory. Given $x \in \N \mathrm{ilp}$, the corresponding fiber in
$\widetilde{\N \mathrm{ilp}}$ is known as a \em Springer fiber\em.
Then Springer theory, in its original form, is concerned with the
action of the Weyl group $W$ on the top-dimensional cohomology
groups of the Springer fibers, as constructed by T.A. Springer in
\cite{S}. Afterwards, several equivalent sheaf theoretic versions of the Weyl group action were constructed. We briefly describe three of them, although there are more. Our ultimate goal is to generalize at least one of these constructions to the global setting.
\subsubsection{Springer Theory via perverse sheaves}
The restriction of $p:\widetilde{\g} \to \g$ to the regular
semisimple locus $\g^{\mathrm{rs}}$ is the $W$-torsor denoted $\widetilde{\g}^{\mathrm{rs}}$. Hence $L := p_{\ast}\Q_{\widetilde{\g}^{\mathrm{rs}}}$ is a $W$-local system on
$\g^{\mathrm{rs}}$. Furthermore, the smallness of the map $p$
implies that $Rp_{\ast}\Q_{\widetilde{\g}}$ is isomorphic to the
intersection cohomology sheaf $IC_{\g}(L)$ associated to the local
system $L$. Finally, by proper base change, the \em Springer sheaf
\em $R\mu_{\ast}\Q_{\widetilde{\N \mathrm{ilp}}}$ is the restriction
of $IC_{\g}(L)$ to $\N \mathrm{ilp}$. The $W$-action on $L$ is then
carried over to the Springer sheaf, which, by the semismallness of
$\mu$, is a perverse sheaf which then decomposes into simple
components according to the regular representation of $W$. See
\cite{L} and \cite{BM}.

\subsubsection{Springer theory via convolution}\label{convolution}
Consider the Steinberg variety $Z := \widetilde{\N \mathrm{ilp}}
\underset{\N \mathrm{ilp}}{\times} \widetilde{\N \mathrm{ilp}}$.
Then $Z$ is a union of conormal bundles in $G/B \times G/B$ indexed
by $W$, whose closures form the irreducible components of $Z$. In
general (i.e., when $\mu$ is replaced by any proper map between
algebraic varieties and $Z$ is the corresponding fiber product), there is an algebra isomorphism
\begin{equation*}
H^{\ast}(Z,\Q) \simeq \Ext^{\ast}(R\mu_{\ast}\Q_{\widetilde{\N
ilp}},R\mu_{\ast}\Q_{\widetilde{\N \mathrm{ilp}}}).
\end{equation*}
However, in the case of the Steinberg variety and Springer
resolution, $Z$ is equidimensional of dimension $\dim_{\C}(\N
\mathrm{ilp})$, and the middle dimension algebra $H(Z) := H^{\dim_{\C}(\N \mathrm{ilp})}(Z,\Q)$ is a subalgebra of $H^{\ast}(Z,\Q)$ which induces the following two algebra isomorphisms:
\begin{equation*}
\Q[W] \simeq H(Z) \simeq \End(R\mu_{\ast}\Q_{\widetilde{\N
ilp}},R\mu_{\ast}\Q_{\widetilde{\N \mathrm{ilp}}}).
\end{equation*}
See \cite[Ch. 3]{CG}.

\subsubsection{Springer theory via nearby cycles}
Let $\psi_{\chi}$ denote the nearby cycles functor associated to the
adjoint quotient map. Then $\psi_{\chi}$ is a functor from (the
derived categories of) constructible sheaves on $\g$ to
constructible sheaves on $\N \mathrm{ilp}$. There is then an
isomorphism
\begin{equation*}
\psi_{\chi}(\Q_{\g}) \simeq R\mu_{\ast}\Q_{\widetilde{\N
\mathrm{ilp}}}.
\end{equation*}
Furthermore, letting $U^{\mathrm{rs}}$ denote the intersection of a small ball around the origin in $\mathfrak{c}$ with the regular semisimple part of $\mathfrak{c}$, there is an action of $\pi_1(U^{\mathrm{rs}})$ on $\psi_{\chi}(\Q_{\g})$.
Although $\pi_1(U^{\mathrm{rs}})$ is isomorphic to the Braid group
of $W$, the action on the Springer sheaf factors through $W$. See
\cite[Sec. 2.2]{Gr}, \cite{Sl}, and \cite{Ho}.

\subsection{What is meant by a global analogue?}\label{ss:global
analog} Returning to diagram \eqref{classical springer diagram}, a
closer examination reveals that all maps appearing in
\eqref{classical springer diagram} are $G$-equivariant and that all
spaces appearing in \eqref{classical springer diagram} are only
considered up to conjugation by $G$. We may thus assert that
ordinary Springer theory is in fact about the study of the
$\C$-points of the associated quotient stacks of \eqref{classical
springer diagram}:
\begin{equation}\label{modern springer diagram}
\begin{diagram}
G/B & \lTo & (\widetilde{\N \mathrm{ilp}}/G)(\C) & \rInto &
(\widetilde{\g}/G)(\C) \\
&& \dTo && \dTo \\
&& (\N \mathrm{ilp}/G)(\C) & \rInto & (\g/G)(\C) \\
&& \dTo && \dTo \\
&& \{ 0 \} & \rInto & \mathfrak{c} \\
\end{diagram}
\end{equation}
The use of the adjective `global' may then be roughly described as
an attempt to create a theory in which the point $\Spec(\C)$ is
replaced by a complex projective curve $X$ which is smooth and
connected. Although it will not suffice to literally replace
$\Spec(\C)$ with $X$, we will still take seriously the idea of
creating ``a family of Springer theories indexed by the curve $X$."

Luckily, global analogues of the Lie algebra $\g$, the adjoint
quotient $\mathfrak{c}$, and the adjoint quotient map $\chi$ have
already received significant study and attention in the form of the
\em Hitchin fibration \em(\cite{Hi})
\begin{equation*}
\chi^{\mathrm{Hit}}:\mathcal{M} \to \mathcal{A}.
\end{equation*}
$\mathcal{M}$ denotes the \em Hitchin moduli stack \em
parameterizing Higgs bundles on $X$, and $\mathcal{A}$ denotes the
\em Hitchin base space. \em Although we have not indicated so in the notation, the definition of $\mathcal{M}$ depends on choosing the
data of a smooth projective curve $X$, a reductive group $G$, and a
line bundle\footnote{In fact, the line bundle is fixed to be the
canonical bundle in the original Hitchin moduli space (\cite{Hi}).}
$\L$ on $X$. When $\L$ is the canonical bundle of $X$, we have an
identification
\begin{equation*}
\mathcal{M} \simeq T^{\ast}\Bun_G(X).
\end{equation*}
We refer the reader to Section \ref{ss:hitchin fibration} for definitions and for a justification of why the Hitchin fibration is a global analogue of $\chi:\g \to \mathfrak{c}$.

In analogy with ordinary Springer theory, we define the \em global
nilpotent cone \em
\begin{equation*}
\N := (\chi^{\mathrm{Hit}})^{-1}(0).
\end{equation*}
Having now identified global analogues for $\N \mathrm{ilp},\g,\mathfrak{c},$ and $\chi$, the obvious goal would
then be to find a resolution of singularities of $\N$, providing a
global analogue of the entirety of \eqref{modern springer diagram}.
This is precisely what we will do in this paper in the particular
case of $G = SL_2$. We hope to be able to extend our construction to $SL_n$ (and possibly arbitrary reductive $G$) in a future paper.

\subsubsection{Relation to Yun's global Springer theory} Recently,
Z. Yun has extensively developed a global Springer theory in
\cite{Y}. Rather than viewing the Hitchin fibration as a global
analogue of the adjoint quotient map, Yun views it as a global analogue of the map in \em local Springer theory \em which is most similar to $\widetilde{\g} \to \g$ in ordinary Springer theory\footnote{There is precedent for this point of view in \cite{Ngo}, in which Hitchin fibers prove to be more manageable than affine Springer fibers.}. While this paper is primarily concerned with the geometry of the global nilpotent cone, Yun's global Springer theory is largely concerned with extending the representation theoretic study of ordinary Springer theory to the global setting. In particular, Yun constructs the \em parabolic Hitchin moduli stack \em $\mathcal{M}^{\mathrm{par}}$, which, besides classifying Higgs bundles $(E,\varphi)$, also adds the additional data of a point $x \in X$ and a choice of $B$-reduction of the fiber $E_x$ which is compatible with $\varphi$. Then the \em parabolic Hitchin fibration
\em
\begin{equation*}
\chi^{\mathrm{par}}:\mathcal{M}^{\mathrm{par}} \to \mathcal{A}
\times X
\end{equation*}
assumes the role of $\widetilde{\g} \to \g$, and an action of the
affine Weyl group is defined on
$R\chi^{\mathrm{par}}_{\ast}\overline{\Q}_{\ell}$. Although we have
claimed that both the Hitchin fibration and parabolic Hitchin
fibration are analogues of $\widetilde{\g} \to \g$, the difference
essentially comes from differing notions of affine (i.e., local)
Springer fiber (\cite{Lu}). Under this analogy, Hitchin fibers
correspond to affine Springer fibers which live in the affine
Grassmannian, while parabolic Hitchin fibers correspond to affine
Springer fibers living in the affine flag variety.

We emphasize, though, that $\mathcal{M}^{\mathrm{par}}$ is only
studied over a subspace of the Hitchin base, which in characteristic $0$ coincides with the anisotropic subspace of $\mathcal{A}$, that
\em does not include the zero section. \em Therefore the global
nilpotent cone plays no role in Yun's global Springer theory.

\subsection{Further motivation}
We mention here two other sources of motivation for the study of
$\N$.
\subsubsection{Geometric Langlands}
So far, one of the few results regarding the global nilpotent cone
is a theorem stating that $\N$ is a Lagrangian substack of
$T^{\ast}\Bun_G(X)$ (\cite{La},\cite{Gi},\cite{Fa}). This has
several important consequences, as proven in \cite{BD}.
\begin{enumerate}
\item The Hitchin fibration is flat of relative dimension $\dim(\N)
= \dim(\Bun_G(X))$.
\item The stack $\Bun_G(X)$ is \em good. \em In particular, this
means that $T^{\ast}\Bun_G(X)$, which is a priori a derived
algebraic stack, is in fact an ordinary algebraic stack.
\item Any $D$-module on $\Bun_G(X)$ whose singular support is contained in $\N$ is
holonomic.
\end{enumerate}
It is the third point that is especially relevant to the Geometric
Langlands Program. Given a smooth projective curve $X$ along with
reductive group $G$ and its Langlands dual $G^{\vee}$, an important
problem in the Geometric Langlands Program is whether one can
associate to any $G^{\vee}$-local system $\mathcal{F}$, a $D$-module $M_{\mathcal{F}}$ on $\Bun_G(X)$, known as the \em Hecke eigensheaf,\em which has eigenvalue $\mathcal{F}$. While this problem is still open in general, A. Beilinson and V. Drinfeld have defined a substack of $\mathrm{LocSys}_{G^{\vee}}(X)$, the stack of
$G^{\vee}$-local systems on $X$, parameterizing so-called
$G^{\vee}$-opers, and for which they were able to construct
corresponding Hecke eigensheaves. The singular support of these
Hecke eigensheaves is the global nilpotent cone $\N$ (\cite{BD}).

\subsubsection{Fundamental Lemma}
The Fundamental Lemma, whose proof was recently completed B.C.
Ng\^{o} in \cite{Ngo} (see also the survey \cite{Na}), is an
identity relating the $\kappa$-orbital integral of an anisotropic
element $\gamma_G$ of a reductive group $G$ defined over a number
field $F$ with the stable orbital integral of the transferred
element $\gamma_H$, which lives in an endoscopic subgroup $H$ of
$G$. Ng\^{o} reduces the proof to the study of the cohomology of
Hitchin fibers. Similarly, there is also a variant known as the Weighted Fundamental Lemma, whose proof was completed by P-H. Chaudouard and G. Laumon (\cite{CL}) using methods similar to those of Ng\^{o}. However, in both cases, it is not necessary to understand all Hitchin fibers, but only those
living over the anisotropic\footnote{For technical reasons, it is
actually necessary in positive characteristic to work with a
slightly smaller subspace.} and generically regular semi-simple subspaces, respectively.

Since $0$ is not an element of the anisotropic or generically regular semisimple loci, neither the proof of the Fundamental Lemma nor of the Weighted Fundamental Lemma takes the global nilpotent cone into consideration. As communicated to the author by Ng\^{o}, it is nevertheless likely that the global nilpotent cone fits into the theory of orbital integrals. Indeed, orbital integrals associated to a regular semisimple element living in a sufficiently small neighborhood of the identity possess an asymptotic expansion in terms of orbital integrals over \em unipotent \em conjugacy classes, whose coefficients are known as \em Shalika germs \em (\cite{Sh}). According to Ng\^{o}, there should be a close relationship between these Shalika germs and the global nilpotent cone.

\subsection{Contents of the paper}
There are two main sections to this paper. In the first section, we
provide the background necessary to construct a global Springer
resolution for $SL_2$. This starts with a brief review of the ordinary Springer resolution, which motivates everything that follows. Global analogues of the nilpotent cone and the flag variety are then introduced. In the former case, this leads us to a discussion of the Hitchin fibration and global nilpotent cone. In the latter case, we review the Drinfeld and Laumon compactifications of the stack of $B$-bundles on $X$. Finally, we end by reviewing the geometry of line bundles and linear systems on $X$. This includes a definition of the stack $C\mathcal{G}(X)$, which later plays a crucial role in understanding and constructing a resolution of singularities of the global nilpotent cone.

The other main section of the paper revolves around the construction of a global Springer resolution, along with the study of its geometric properties. As an intermediate step, we first define the \em partial global Springer resolution \em $\widehat{\N}$. We then verify that the projection map from $\widehat{\N}$ to the global nilpotent cone is a proper, birational equivalence, which is furthermore a resolution of singularities when the genus of $X$ is at most one. In general, we are able to understand the singularities of $\widehat{\N}$ by producing a smooth map to $C\mathcal{G}(X)$ and studying its singularities instead. By pulling back a resolution of singularities of $C\mathcal{G}(X)$ to $\widehat{\N}$, we are finally able to define the \em global Springer resolution \em $\widetilde{\N}$. This allows us to obtain corollaries on the equidimensionality of the global nilpotent cone along with an enumeration of its irreducible components. We end the paper with a discussion of the stable part of the global nilpotent cone, as well as a description of $\widetilde{\N}$ in the case where the twisting line bundle is of low degree.

\subsection{Conventions and notation}
\subsubsection{Conventions}
We work over the complex numbers $\C$. All geometric objects
(schemes, stacks) discussed will be defined over $\C$. Similarly,
all algebraic groups will also be defined over $\C$.

We will work heavily throughout this paper with algebraic stacks. We
will think of an algebraic stack $Y$ as functor
\begin{equation*}
Y:\mathrm{Schemes}/\C \to \mathrm{Groupoids}
\end{equation*}
from the category of schemes over $\C$ to the category of groupoids.
Given a scheme $S$, the groupoid of $S$-points of $Y$ will be
denoted $Y(S)$. We will typically only describe the objects of
$Y(S)$, as it should be clear what the morphisms are. Given a
commutative $\C$-algebra $R$, we will also sometimes discuss the
$R$-points of $Y$, by which we will mean the $\Spec(R)$-points of
$Y$. For the types of stacks discussed in this paper, we recommend
\cite{He} for a concise introduction.

Although we will typically discuss the $S$-points of $Y$, we note
that it actually suffices to only consider $R$-points, using the
fact that $Y$ is a sheaf and the fact that any scheme is a limit of
affine schemes. Furthermore, since any ring is a colimit of finitely
generated rings, we may assume that $S$ is a scheme of finite type
over $\C$. This fact will be useful in allowing us to use scheme
theoretic results with finite type hypotheses, which we will do
implicitly without further mention.

Lastly, we note that there is no difference between the geometric
points of $S$ and the closed points of $S$, as $S$ is defined over
an algebraically closed field. When stating definitions and results
that apply more generally to geometric points of $S$, we will refer
to them as geometric points. Likewise for closed points.

\subsubsection{Notation}
$X$ will denote a smooth, projective, connected algebraic curve of
genus $g$, and $\L$ will denote a fixed line bundle on $X$. In
Section \ref{s:background}, there will be no assumptions placed on
the degree of $\L$, while in Section \ref{s:resolution} we will
assume that the degree of $\L$ is even and no less than $2g$ (with
the exception of Section \ref{ss:lower degree}).

Since we will make several infinitesimal arguments throughout the
paper, we will use the notation $S[\epsilon]$ (resp., $R[\epsilon]$)
to denote the product of the scheme $S$ with the dual numbers
(resp., the tensor product of the ring $R$ with the ring of dual numbers).

As the $S$-points of the stacks we work with will typically
parameterize $S$-families of objects on $X$, the product $X \times
S$ will be conveniently denoted by $X_S$. It is equipped with
projection maps
\begin{diagram}
X_S & \rTo^{\pi} & X\\
\dTo^p\\
S
\end{diagram}
Most of the objects we work with in this paper will be vector
bundles, and more generally coherent sheaves, on $X_S$. Given a
point $x \in X_S$ and a coherent sheaf $\mathcal{F}$ on $X_S$, we
will let $\mathcal{F}_x$ denote the \em fiber \em of $\mathcal{F}$
at $x$ (as opposed to the stalk). We use instead the notation
$i_x^{\ast}\mathcal{F}$ to denote the stalk of $\mathcal{F}$ at $x$.
We will only have occasion to discuss stalks in the proof of
Proposition \ref{defect S}.

We let $G$ denote a complex reductive algebraic group whose derived
subgroup $[G,G]$ is simply connected. The group $B$ denotes a Borel
subgroup of $G$, and $T$ denotes a maximal torus of $G$, which will
frequently be assumed to be $B/[B,B]$. For any group $H$, we use the
notation $\Bun_H(X)$ for the stack of principal $H$-bundles on $X$.
When dealing with principal $H$-bundles, recall that there is a
vector bundle associated to any choice of an $H$-bundle and a
representation of $H$. Given an $H$-bundle $E$ and a representation
$V$ of $H$, the associated vector bundle $E \overset{H}{\times}
V$ is denoted by $V_E$. In the special case of the adjoint action of $H$ on its Lie algebra $\mathfrak{h}$, we let $\mathrm{ad}(E) := E
\overset{H}{\times} \mathfrak{h}$.

Finally, we will frequently be in the situation of a group $H$
acting on a space $Y$. We will use the notation $Y/H$ for the
stack-theoretic quotient, and the notation $Y/ \! /H$ for the GIT
quotient.

\subsection*{Acknowledgements}
I would first like to thank S. Gunningham, O. Gwilliam, I. Le, and T. Stadnik for many useful conversations, and for their sustained interest in the topic of this paper. I am very grateful to D. Treumann for his consistently helpful insights and wisdom. I am also indebted to M. Emerton for graciously sharing his seemingly endless knowledge. I thank E. Zaslow for his detailed comments and corrections on a previous draft of this paper, as well as E. Getzler for for interesting comments and questions.

Most of all, I would like to thank my adviser, D. Nadler. I am
grateful for his consistent guidance, encouragement, insight, and
consummate mentorship. I am furthermore indebted to him for exposing
me to a beautiful array of mathematical ideas and modes of thought,
as well as for providing the initial stimulus to consider the topic
of this paper.

\tableofcontents

\section{Background}\label{s:background}

We begin this section with a brief review of the ordinary Springer
resolution of the nilpotent cone $\N \mathrm{ilp} \subset \g$. We
then review basic facts and definitions regarding the Hitchin
fibration and the global nilpotent cone $\N$.  After this we review
the stack $\Bun_B(X)$ and its two relative compactifications over
$\Bun_G(X)$, due to Drinfeld and Laumon. Finally, we end the section
by reviewing some classical geometry related to varieties of line
bundles and linear systems on curves, as well as their
stack-theoretic counterparts.

\subsection{The Springer resolution}\label{ss:springer resolution}

We begin by giving the definition of the nilpotent cone $\N
\mathrm{ilp} \subset \g$.  In order to do so in a way that
generalizes nicely to the global situation, we must first recall the
adjoint quotient map.

\begin{dfn}
We call $\mathfrak{c} := \g/ \! /G$ the \em adjoint quotient \em of
$\g$ and the corresponding map $\chi:\g \to \mathfrak{c}$ the \em
adjoint quotient map\em.
\end{dfn}

Recall that when $\g = \mathfrak{gl}_n,$ the adjoint quotient map
may be viewed as the map sending a matrix $A$ to the non-leading
term coefficients of its characteristic polynomial. Therefore $A^n =
0$ if and only if $\chi(A) = 0$.  This leads to the definition of
the nilpotent cone.

\begin{dfn}
$\N \mathrm{ilp} := \chi^{-1}(0)$ is called the \em nilpotent cone
\em of $\g$.
\end{dfn}

In order to properly make sense of the definition of $\N
\mathrm{ilp}$, we recall a result of Chevalley which states that
$\mathfrak{c} \simeq \mathbb{A}^r$, where $r$ denotes the rank of
$\g$ (\cite{Bour}). More precisely, this isomorphism is realized by
the existence of $r$ independent generators of the ring of functions of $\mathfrak{c}$. The generators are each homogeneous of degrees
$d_1, \ldots, d_r$. Letting $e_i$ denote the $i^{\mathrm{th}}$
exponent of $\g$, we have the relation
\begin{equation*}
d_i = e_i + 1.
\end{equation*}

It is straightforward to see from the definition that $\N
\mathrm{ilp}$ is a singular affine variety (for example, the origin
is always singular).  The singularities of $\N \mathrm{ilp}$ have
been completely classified in Lie theoretic terms, as follows.

\begin{dfn}
An element $x \in \g$ is said to be \em regular \em if its
centralizer $Z_{\g}(x)$ has the minimal possible dimension $r$ ($= \rk(\g)$).  We let $\g^{\mathrm{reg}} \subset \g$ denote the locus
of regular elements.
\end{dfn}

\begin{prop}\label{reg sm locus}
The smooth locus of $\N \mathrm{ilp}$ is given by the regular
elements of $\N \mathrm{ilp}$. Furthermore, $\N \mathrm{ilp}$ is a
normal algebraic variety.
\end{prop}
\begin{proof}
See \cite[Sec. 10.3]{HTT}.
\end{proof}

In order to give some motivation for constructing a resolution of
singularities of $\N \mathrm{ilp}$ using Proposition \ref{reg sm
locus}, we recall that for a regular nilpotent $x \in \g$, there is
a unique Borel subalgebra containing $x$. If we identify the flag
variety $G/B$ with the variety of Borel subalgebras of $\g$, we may
define
\begin{equation*}
\widetilde{\N \mathrm{ilp}} := \{(x,\mathfrak{b}) \in \N
\mathrm{ilp} \times G/B: x \in \mathfrak{b} \}.
\end{equation*}

The projection map $\widetilde{\N \mathrm{ilp}} \to \N \mathrm{ilp}$
is proper (since $G/B$ is compact) and is a birational equivalence
(inducing an isomorphism over the regular part of $\N ilp$).
Finally, $\widetilde{\N \mathrm{ilp}}$ is smooth because it is a
vector bundle over $G/B$.

\begin{thm}
$\widetilde{\N \mathrm{ilp}} \to \N \mathrm{ilp}$ is a resolution of
singularities. Furthermore, $\widetilde{\N \mathrm{ilp}} \simeq
T^{\ast}(G/B)$.
\end{thm}
\begin{proof}
See \cite[Sec. 3.2]{CG}.
\end{proof}

The resolution of singularities $\widetilde{\N \mathrm{ilp}} \to \N
\! \mathrm{ilp}$ is known as the \em Springer resolution\em.

\begin{rem}\label{SL}
When $G = SL_n$, we wish to emphasize an alternative but equivalent
formulation of $\widetilde{\N \mathrm{ilp}}$.  Rather than
identifying $\mathrm{Fl}_n := SL_n/B$ with the variety of Borel
subalgebras, we may instead identify it with the variety of flags
$V_1 \subset V_2 \subset \ldots \subset V_n$ in which $V_i$ is a
vector space of dimension $i$. It is then straightforward to check
that the variety of pairs $(A,(V_1 \subset \ldots \subset V_n)) \in
\N \mathrm{ilp} \times \mathrm{Fl}_n$ such that $A \cdot V_i \subset
V_i$ for all $i$ is the same as $\widetilde{\N \mathrm{ilp}}$.
Furthermore, the condition that $A \cdot V_i \subset V_i$ is
equivalent to $A \cdot V_i \subset V_{i-1}$ since $A$ is nilpotent.
\end{rem}

\subsection{The Hitchin fibration and global nilpotent
cone}\label{ss:hitchin fibration}

In introducing the Hitchin fibration and global nilpotent cone, we
would like to emphasize their analogy with finite dimensional Lie
algebras and classical Springer theory as reviewed in Section
\ref{ss:springer resolution}. As discussed in Section \ref{ss:global analog}, we may obtain a first approximation to what we will call
`global Springer theory' by replacing the point $\Spec(\C)$ by the
curve $X$ in \eqref{modern springer diagram}. We begin by defining
the global analogue of the Lie algebra $\g$.

\begin{dfn}\label{hitchin moduli def}
Given the group $G$ and curve $X$, together with the additional data of a line bundle $\L$ on $X$, we define, following \cite{Ngo}, the
\em Hitchin moduli stack \em $\mathcal{M}_{X,G,\L}$ to be the
mapping stack $\Hom(X,\g_{\L}/G)$.  Recall that $\g_{\L}/G$ denotes the stack quotient and that $\g_{\L}$ denotes the associated vector bundle given by viewing $\g$ as a representation of $\mathbb{G}_m$. We will typically denote $\mathcal{M}_{X,G,\L}$ by $\mathcal{M}$ if the triple $(X,G,\L)$ is clear from the context.
\end{dfn}

\begin{rem}\label{global lie algebra}
Note that, after a choice of trivialization, we can roughly think of a $\C$-point of $\mathcal{M}$ as a collection of elements $\varphi_x \in \g$ for every point $x \in X$.
\end{rem}

While Definition \ref{hitchin moduli def} is conceptually useful for the transition to a global Springer theory, the following lemma
provides a more standard definition of the Hitchin moduli stack.

\begin{lem}\label{alternative hitchin lemma}
$\mathcal{M}$ is equivalent to the stack whose $S$-points consist of all pairs $(E,\varphi)$ where
\begin{enumerate}
\item  $E$ is a principal $G$-bundle on $X_S$ and
\item  $\varphi \in H^0(X_S,\mathrm{ad}(E) \otimes \pi^{\ast}\L)$.
\end{enumerate}
\end{lem}
\begin{proof}
The proof is essentially tautological.  Recall that an $S$-point of
$\Hom(X,\g_{\L}/G)$ is the same as an $X_S$ point of $\g_{\L}/G$.
\end{proof}

The pair $(E,\varphi)$ appearing in Lemma \ref{alternative hitchin
lemma} will be referred to as a \em Higgs bundle\em, and the section $\varphi$ will be referred to as a \em Higgs field\em.

\begin{exmp}
Let us give an equivalent formulation of a Higgs bundle when $G =
SL_n$ which will be useful in subsequent sections. A principal
$SL_n$-bundle is equivalent to a rank $n$ vector bundle $E$ together with an isomorphism $\det(E) \simeq \O_X$. A Higgs field is then
equivalent to giving a twisted endomorphism $\varphi:E \to E \otimes \L$ such that $\tr(\varphi) = 0.$  We will conflate these equivalent notions of Higgs bundle whenever $G = SL_n$.
\end{exmp}

Having found a suitably global version of the Lie algebra $\g$, we now formulate a global version of the adjoint quotient $\mathfrak{c}$ and the adjoint quotient map $\chi$.

\begin{dfn}
Define the \em Hitchin base space \em $\mathcal{A}_{X,G,\L}$ (or
simply $\mathcal{A}$ if the context is clear) to be the space of
global sections of the affine bundle $\mathfrak{c}_{\L} := \L
\overset{\mathbb{G}_m}{\times} \mathfrak{c}$ on $X$.
\end{dfn}

The following lemma gives a more concrete (but less canonical)
description of $\mathcal{A}$. This description of $\mathcal{A}$ can
be originally found in \cite{Hit}.

\begin{lem}
Recalling the notation from Section 1.1, there is a non-canonical
isomorphism
\begin{equation}\label{hitchin base}
\mathcal{A} \simeq
\underset{i=1}{\overset{r}{\bigoplus}}H^0(X,\L^{\otimes d_i}).
\end{equation}
\end{lem}

\begin{proof}
Choosing generators $f_1,\ldots,f_r$ of degrees $d_1, \ldots, d_r$
for the ring of functions of $\mathfrak{c}$ determines an action of
$\mathbb{G}_m$ on $\mathfrak{c}$, from which it is easily checked
that $\mathfrak{c}_{\L} \simeq
\underset{i=1}{\overset{r}{\bigoplus}}\L^{\otimes d_i}$.  This gives
the isomorphism of \eqref{hitchin base}, which is non-canonical due
to the choice of generators $f_1,\ldots,f_r$.
\end{proof}

In order to construct a global analogue of the map $\chi$, notice that $\chi$ is both $\mathbb{G}_m$-equivariant and $G$-invariant.
Therefore $\chi$ induces a map $\g_{\L}/G \to \mathfrak{c}_{\L}$,
from which we obtain a map
\begin{equation*}
\chi^{\mathrm{Hit}}:\mathcal{M} \to \mathcal{A}
\end{equation*}
known as the \em Hitchin fibration\em.

\begin{rem}
Following up on the informal commentary of Remark \ref{global lie
algebra}, the map $\chi^{\mathrm{Hit}}$ may be roughly thought of as associating to a Higgs bundle $(E,\varphi)$ the collection of
elements $\chi(\varphi_x) \in \mathfrak{c}$ for every point $x \in
X$.
\end{rem}

\begin{exmp}
In the particular case of $G = SL_n,$ the degrees of the generators
of $\mathfrak{c}$ are given by $2,\ldots,n,$ which are exactly the
degrees of the coefficients of the characteristic polynomial of an
element $A \in \mathfrak{sl}_n$, viewed as symmetric polynomials in
the eigenvalues of $A$.  Then, as alluded to above,
$\chi^{\mathrm{Hit}}:\mathcal{M} \to \mathcal{A}$ can be thought of
as a global characteristic polynomial map. Indeed, for any point $x
\in X$, the fiber $\chi^{\mathrm{Hit}}(E,\varphi)_x$ is nothing but
the characteristic polynomial of $\varphi_x$.  The absence of
$H^0(X,\L)$ appearing as a direct summand of $\mathcal{A}$ comes
from the condition that $\tr(\varphi) = 0$.
\end{exmp}

Having formulated a global analogue of $\chi$, we may now define the global nilpotent cone, first introduced in \cite{Lau} and \cite{La}.

\begin{dfn}
The \em global nilpotent cone \em $\N_{X,G,\L}$ (typically just
written as $\N$) is defined to be the reduced substack of
$(\chi^{\mathrm{Hit}})^{-1}(0)$, the fiber over zero of the Hitchin
fibration.  It is therefore a closed substack of the Hitchin moduli
stack.  A pair $(E,\varphi) \in \N$ will be referred to as a
nilpotent Higgs bundle, and $\varphi$ will be called a nilpotent
Higgs field.
\end{dfn}

\begin{exmp}
When $G = SL_n$, a Higgs bundle $(E,\varphi)$ is nilpotent if and
only if $\varphi^n = 0$, where $\varphi^n = (\phi \otimes
\mathrm{id}_{\L^{\otimes n-1}}) \circ \ldots \circ (\varphi \otimes
\mathrm{id}_{\L}) \circ \varphi$ is a map $E \to E \otimes
\L^{\otimes n}$.
\end{exmp}

Lastly, it will be useful to generalize the notion of a regular
element of $\g$ to the global setting.

\begin{dfn}\label{regular}
The \em regular \em locus of the Hitchin moduli stack is the
substack of $\mathcal{M}$ defined to be $\mathcal{M}^{\mathrm{reg}}
:= \Hom(X,\g^{\mathrm{reg}}_{\L}/G)$.  A Higgs bundle $(E,\varphi)
\in \mathcal{M}(S)$ will be called a \em regular Higgs bundle \em
(or just \em regular\em) if $(E,\varphi) \in \mathcal{M}^{\mathrm{reg}}(S)$. Furthermore, if $(E,\varphi) \in
\mathcal{M}(S)$ corresponds to a map $h:X_S \to \g_{\L}/G$, we will
say that $(E,\varphi)$ is \em generically regular \em if there
exists an open set $U \subset X$ such that the restriction of $h$ to $U_S$ maps to $\g^{\mathrm{reg}}_{\L}/G$.
\end{dfn}

\subsection{A relative compactification of $\Bun_B(X)$}\label{ss:B
compactification}

The goal of this section is to formulate an appropriate global
analogue of the flag variety $G/B$.  Let us begin by returning to
diagram (\ref{classical springer diagram}) and considering the space of maps $\Hom(X,G/B)$.  In order to give an explicit description of
$\Hom(X,G/B)$, we will use the Pl\"{u}cker embedding
\begin{equation*}
G/B\hookrightarrow \underset{i=1}{\overset{r}{\prod}}
\mathbb{P}(V_{\omega_i}^{\ast}),
\end{equation*}
where $V_{\omega_i}$ is the fundamental representation of $G$
associated to the fundamental weight $\omega_i$.  Since a map $X \to \mathbb{P}^n$ of degree $d$ is the same as specifying a line bundle
$\lambda$ of degree $-d$ with an embedding $0 \to \lambda
\to \O_X^{n+1}$ such that the quotient is locally free, we obtain the following Pl\"{u}cker description of $\Hom(X,G/B)$.

Giving a map $X \to G/B$ is equivalent to giving a collection of line \em subbundles \em
\begin{equation}\label{plucker}
\{\lambda_{\mu} \hookrightarrow \O_X \otimes V_{\mu}\}
\end{equation}
for every dominant weight $\mu$ satisfying the Pl\"{u}cker
relations\footnote{We omit the details of the Pl\"{u}cker relations.
The interested reader may consult \cite{Ku}.}.  A consequence of the
Pl\"{u}cker relations is that it suffices to specify line subbundles
only for the finitely many fundamental weights of $G$.

On the other hand, when $G = SL_n$, there is an alternative,
flag-like description of $\Hom(X,\mathrm{Fl}_n)$. In this case,
giving a map $X \to \mathrm{Fl}_n$ is equivalent to giving a flag of
\em subbundles \em
\begin{equation*}
V_1 \subset \ldots \subset V_{n-1} \subset \mathcal{O}_X^n
\end{equation*}
in which $\rk(V_i) = i$.

Even though the flag variety is complete, the space of maps $\Hom(X,G/B)$ is not.  In order to correct this defect, there is a compactification of $\Hom(X,G/B)$ due to Drinfeld known as the space of \em quasi-maps \em (\cite{Ku},\cite{Br}).  This space is obtained by taking the Pl\"{u}cker description of $\Hom(X,G/B)$ in \eqref{plucker} and requiring only that the $\lambda_{\mu}$ be \em subsheaves \em of $\O_X \otimes V_{\mu}$.  This means that the cokernel of $\lambda_{\mu} \hookrightarrow \O_X \otimes V_{\mu}$ may have torsion.

In a similar fashion, when $G = SL_n$, there is a compactification
of $\Hom(X,\mathrm{Fl}_n)$ due to Laumon given by considering flags
of \em subsheaves \em $V_1 \subset \ldots \subset V_{n-1} \subset
\mathcal{O}_X^n$. The resulting space is known as the space of \em
quasi-flags, \em which coincides with the space of quasi-maps if and only if $n=2$ (\cite{Ku}, \cite{Br}). A. Kuznetsov has shown that
the space of quasi-flags is a small resolution of singularities of
the space of quasi-maps in \cite{Ku}.

\begin{exmp}
Let us examine the compactification of
$\Hom(\mathbb{P}^1,\mathrm{Fl}_2)$, in which the Drinfeld and Laumon
compactifications coincide.  A degree one map $\mathbb{P}^1 \to
\mathrm{Fl}_2$ is equivalent to giving a subbundle
\begin{equation*}
0 \to \O(-1) \to \O \oplus \O.
\end{equation*}
Writing $z,w$ for the homogeneous coordinates of
$\mathbb{P}^1$, such a map may be written as
\begin{equation*}
\begin{pmatrix} az+bw \\
cz+dw \end{pmatrix}, \text{ with }a,b,c,d \in \C.
\end{equation*}
The condition that $\O(-1)$ be a subbundle of $\O \oplus \O$ is
equivalent to requiring that
\begin{equation*}
\det\begin{pmatrix}a & b \\ c & d
\end{pmatrix} \neq 0.
\end{equation*}
Therefore the space of maps $\Hom(\mathbb{P}^1,\mathrm{Fl}_2)$ is
isomorphic to the complement of the Pl\"{u}cker embedding of
$\mathbb{P}^1 \times \mathbb{P}^1$ inside $\mathbb{P}^3$.  The
corresponding space of quasi-maps/quasi-flags is simply
\begin{equation*}
\mathbb{P}\Hom(\O(-1),\O^2) \simeq \mathbb{P}^3.
\end{equation*}
\end{exmp}

For the purposes of finding a global Springer resolution,
$\Hom(X,G/B)$ and its compactifications will be insufficient. The problem is that in the Pl\"{u}cker (and flag-like, when $G =
SL_n$) description of $\Hom(X,G/B)$, all line bundles are subbundles of a trivial vector bundle.  The Hitchin moduli stack, on the other hand, lives over the entire moduli of $G$-bundles. The solution to the problem is to place $\Hom(X,G/B)$ and its compactifications into the larger context of principal $B$-bundles.

To begin, there is a map of stacks
\begin{equation*}
q:\Bun_B(X) \to \Bun_G(X)
\end{equation*}
arising from the inclusion $B \subset G$. It is then straightforward
to check that the fiber over the trivial $G$-bundle $E_G^0$ is
exactly $\Hom(X,G/B)$. Then, just as $\Hom(X,G/B)$ is not complete,
the map $q$ is not proper.  We would therefore like to have a
relative compactification of $\Bun_B(X)$ so that the corresponding
fiber over $E_G^0$ coincides with one of the compactifications of
$\Hom(X,G/B)$.

Relative compactifications generalizing quasi-maps and quasi-flags
exist, and are known as the Drinfeld and Laumon compactifications,
respectively (\cite{BG}). The Drinfeld compactification, denoted
$\Bun_B^D(X)$, exists for any reductive algebraic group $G$, and is
a generalization of the space of quasi-maps. The Laumon
compactification, denoted $\Bun_B^L(X)$, only exists when $G=SL_n$,
and is a generalization of the space of quasi-flags.

\begin{dfn}\label{drinfeld compactification}
The \em Drinfeld compactification \em $\Bun_B^D(X)$ of $\Bun_B(X)$
is the algebraic stack whose $S$-points are given by pairs
$(\mathcal{F}_T,\mathcal{F}_G)$, where $\mathcal{F}_T$ is a
$T$-bundle and $\mathcal{F}_G$ is a $G$-bundle. Furthermore, we
require that for every dominant weight $\mu$, there is an embedding
of coherent sheaves
\begin{equation*}
L^{\mu}_{\mathcal{F}_T} \hookrightarrow V^{\mu}_{\mathcal{F}_G}.
\end{equation*}
The collection of embeddings for every dominant weight $\mu$ is
required to satisfy the Pl\"{u}cker relations. See \cite{BG} for a
full description.
\end{dfn}

\begin{dfn}\label{laumon compactification}
The \em Laumon compactification \em $\Bun_B^L(X)$ of $\Bun_B(X)$ is
the algebraic stack whose $S$-points consist of flags of coherent
sheaves $V_1 \subset V_2 \subset \ldots \subset V_n$ on $X_S$ such
that:
\begin{itemize}
\item  $V_i$ is a vector bundle of rank $i$ for $1\leq i \leq n-1$
and $V_n$ is an $SL_n$-bundle.
\item  Each quotient $V_i/V_{i-1}$ is $S$-flat.
\end{itemize}
\end{dfn}

\begin{rem}\label{difference}
The difference between an $S$-point of $\Bun_B^L(X)$ and an
$S$-point of $\Bun_B(X)$ is that an $S$-point of $\Bun_B(X)$
consists of a flag of vector bundles as in Definition \ref{laumon
compactification} such that each quotient $V_i/V_{i-1}$ is
$X_S$-flat.  This is equivalent to saying that each $V_{i-1}$ is a
subbundle of $V_i$. Note that when $S = \Spec(\C)$, each quotient is
automatically $S$-flat, so each $V_{i-1}$ is simply a coherent
subsheaf of $V_i$.
\end{rem}

The following proposition shows that the compactifications
$\Bun_B^D(X)$ and $\Bun_B^L(X)$ each have desirable properties.

\begin{prop}\label{compactification}
The following properties hold for $\Bun_B^D(X)$ (for any reductive
algebraic group $G$) and for $\Bun_B^L(X)$ (when $G=SL_n$).
\begin{enumerate}
\item  There is a natural inclusion map $j^D:\Bun_B(X) \to \Bun_B^D(X)$ (resp., $j^L:\Bun_B(X) \to \Bun_B^L(X)$) making $\Bun_B(X)$ an open, dense
substack.
\item  There is a proper map $q^D:\Bun_B^D(X) \to \Bun_G(X)$ (resp.,
$q^L:\Bun_B^L(X) \to \Bun_G(X)$) such that $q^D \circ j^D = q$
(resp., $q^L \circ j^L = q$).
\end{enumerate}
\end{prop}
\begin{proof}
The density statement may be found in \cite[Prop 1.2.3]{BG}. We
remark that the simply-connectedness assumption on $[G,G]$ is
necessary here. A proof that $q^D$ is proper may be found in
\cite[Prop 1.2.2]{BG}.
\end{proof}

For later use, we now record some basic information about
$\Bun_B(X)$, and hence its compactifications by the density
statement of Proposition \ref{compactification}.  Letting $T :=
B/[B,B]$ be the maximal torus of $G$, there is an induced map
\begin{equation*}
r:\Bun_B(X) \to \Bun_T(X),
\end{equation*}
which may be extended to maps
\begin{align*}
r^D:\Bun_B^D(X)& \to \Bun_T(X),\\
r^L:\Bun_B^L(X)& \to \Bun_T(X).
\end{align*}
The map $r^D$ is the obvious projection map. The map $r^L$ is given by associating to a flag $V_1 \subset \ldots \subset V_n$ the $(n-1)$-tuple of line bundles
\begin{equation*}
(V_1,\det(V_2), \ldots, \det(V_{n-1})).
\end{equation*}

Each of these maps induce bijections on connected components. Let us therefore review why the connected components of $\Bun_T(X)$ are in bijection with the coweight lattice $\Lambda_G$ of $G$. Given a $T$-bundle $\mathcal{F}_T$ on $X$ and a divisor
\begin{equation*}
D = \sum \alpha_kx_k
\end{equation*}
on $X$ whose coefficients $\alpha_k$ are coweights, we may construct a new $T$-bundle $\mathcal{F}_T(D)$ which is defined by the property that for every weight $\mu$, there is an equality
\begin{equation*}
L^{\mu}_{\mathcal{F}_T(D)} = L^{\mu}_{\mathcal{F}_T}(\sum \langle
\alpha_k,\mu \rangle x_k).
\end{equation*}
We will say that $\mathcal{F}_T$ has degree $\alpha \in \Lambda_G$ if for every weight $\mu$,
\begin{equation*}
\deg(L^{\mu}_{\mathcal{F}_T}) = \langle \alpha, \mu \rangle.
\end{equation*}
Equivalently, $\mathcal{F}_T$ has degree $\alpha$ if and only if $\mathcal{F}_T = \mathcal{F}_T^0(D)$, where $\mathcal{F}_T^0$ is the trivial $T$-bundle and $D$ is a divisor whose coefficients sum to $\alpha$. This discussion is recorded in the following proposition.

\begin{prop}
$\pi_0(\Bun_B^L(X)) \simeq \pi_0(\Bun_B^D(X)) \simeq
\pi_0(\Bun_B(X)) \simeq \pi_0(\Bun_T(X)) \simeq \Lambda_G$.
\end{prop}

To compute the dimension of irreducible components of $\N$, we will
need to first compute the dimension of the connected components of
$\Bun_B(X)$. The following proposition is similar to \cite[3.5]{FM}.

\begin{prop}\label{dimension B}
Let $\Bun_{B,\alpha}(X)$ denote the connected component of
$\Bun_B(X)$ corresponding to $\alpha \in \Lambda_G$.  Then the
dimension of $\Bun_{B,\alpha}(X)$ is given by $-2|\alpha| +
\dim(B)(g-1)$.
\end{prop}
\begin{proof}
Since $\Bun_B(X)$ is a smooth stack, to compute the dimension of
$\Bun_{B,\alpha}(X)$ it suffices to pick any $B$-bundle
$\mathcal{F}_B$ on $X$ of degree $\alpha$ and compute the dimension
of the naive tangent complex\footnote{See Section \ref{ss:divisor
geometry} for more details.} at $\mathcal{F}_B$. By standard
deformation theory,
\begin{equation*}
\dim_{\mathcal{F}_B}(\Bun_{B,\alpha}(X)) =
-\chi(X,\mathfrak{b}_{\mathcal{F}_B}).
\end{equation*}
Since $\mathfrak{b}_{\mathcal{F}_B}$ a vector bundle of rank
$\dim(B)$, it suffices by Riemann-Roch to compute the degree of
$\mathfrak{b}_{\mathcal{F}_B}$.  To this end, we may assume that
$\mathcal{F}_B$ has a $T$-reduction coming from some $T$-bundle
$\mathcal{F}_T$ of degree $\alpha$.  Then
$\mathfrak{b}_{\mathcal{F}_B} = \mathcal{F}_T
\overset{T}{\times}\mathfrak{b}$.  Since
\begin{equation*}
\mathfrak{b} = \underset{\theta \in
\mathcal{R}^+}{\bigoplus}\g_{\theta}
\end{equation*}
is a direct sum of positive root spaces, we see that
\begin{equation*}
\deg(\mathcal{F}_T \overset{T}{\times}\mathfrak{b}) =
\underset{\theta \in \mathcal{R}^+}{\sum}\langle \alpha,\theta
\rangle = 2|\alpha|.
\end{equation*}
\end{proof}

\begin{cor}
The dimension of $\Bun^D_{B,\alpha}(X)$ and of
$\Bun^L_{B,\alpha}(X)$ is given by $-2|\alpha| + \dim(B)(g-1)$.
\end{cor}
\begin{proof}
This is a direct consequence of the density statement of Proposition
\ref{compactification} and from Proposition \ref{dimension B}.
\end{proof}

\subsection{Geometry of linear systems on a curve}\label{ss:divisor geometry}

In this section we will review some classic results about line
bundles and linear systems on curves.  The main reference for this
section is \cite{ACGH}.  We will assume that the genus of $X$ is at
least two. We will largely be concerned with the following two
classical varieties (and their stack-theoretic counterparts).
\begin{align*}
W_d^r(X)& = \{ \lambda \in \Pic_d(X): h^0(X,\lambda) \geq r+1 \}\\
G_d^r(X)& = \{g_d^r\text{'s on }X\}
\end{align*}

Recall that $g_d^r$ is the classical notation for a degree $d$, rank
$r$ linear system on $X$.  Therefore, in less concise notation,
$G_d^r(X)$ is the variety of pairs $(\lambda,V)$ where $\lambda \in
\Pic_d(X)$ and $V \subset H^0(X,\lambda)$ is a subspace of dimension
$r+1$.  For the purposes of global Springer theory, we will be
especially interested in $G_d^0(X)$.  The two varieties $W_d^r(X)$
and $G_d^r(X)$ are related by the fact that the scheme-theoretic
image of $G_d^r(X)$ under the projection map $G_d^r(X) \to
\Pic_d(X)$ is precisely $W_d^r(X)$. The next result summarizes the
crucial geometric properties of $W_d^0(X)$ and $G_d^0(X)$.

\begin{thm}\label{grd geometry}
(i) $G^0_d(X)$ is smooth of dimension $d$ for all $d$.\\
(ii) $W^0_d(X)$ is reduced, irreducible, normal, and Cohen-Macaulay
of dimension $d$.  If $d<g$, then the singular locus of $W^0_d(X)$
is $W^1_d(X)$.
\end{thm}
\begin{proof}
See \cite[Cor. 4.5]{ACGH}.
\end{proof}

The assumption that $d<g$ in Theorem \ref{grd geometry} is present simply because $W^0_d(X)$ is smooth whenever $d \geq g$. Indeed, given $\lambda \in \Pic_d(X)$ with $d \geq g$, the Riemann-Roch formula implies that
\begin{equation*}
\chi(X,\lambda) = d+1-g \geq 1.
\end{equation*}
Therefore $W^0_d(X) = \Pic_d(X)$ is smooth.

In Section \ref{ss:higher genus}, we will need variants of $W_d^0(X)$ and $G_d^0(X)$, which we now discuss.  First, recall that the Picard stack is obtained by taking the stack quotient $\Pic(X)/\mathbb{G}_m$ of the Picard variety by a trivial action of $\mathbb{G}_m$.  To avoid confusion, the Picard stack will be denoted $\Bun_{\mathbb{G}_m}(X)$. There are then stack-theoretic versions of $W_d^r(X)$ and $G_d^r(X)$:
\begin{equation*}
\begin{split}
\mathcal{W}_d^r(X)& = W_d^r(X)/\mathbb{G}_m\\
\mathcal{G}_d^r(X)& = G_d^r(X)/\mathbb{G}_m
\end{split}
\end{equation*}

To have a good understanding of the partial global Springer resolution defined in Section \ref{ss:construction}, it is necessary to give a precise formulation of the $S$-points of $\mathcal{W}_d^r(X)$ and $\mathcal{G}_d^r(X)$, as described in \cite[Sec. 4.3]{ACGH}. In order to do so, we now review basic properties of Fitting ideals, which may be found in \cite[Appendix]{MW}. To begin, assume that $M$ is a finitely presented module over a commutative ring $R$. Given a presentation \begin{equation*}
R^{\oplus a} \overset{A}{\to} R^{\oplus b} \to M \to 0,
\end{equation*}
define, for $h \geq 0$, the $h^{\mathrm{th}}$ \em Fitting ideal \em of $M$ to be the ideal of $R$ generated by the $(b-h) \times (b-h)$ minors of $A$. It is denoted $F^h(M)$ or $F^h_R(M)$. We adopt the convention that $F^h(M) = R$ if $a < b-h$. The following properties of $F^h(M)$ are well-known.
\begin{enumerate}
\item $F^h(M)$ is independent of the presentation chosen for $M$.
\item If $N$ is also an $R$-module, then $F^h(M \oplus N) = F^h(M)F^h(N)$.
\item Fitting ideals are stable under base change in the sense that if $A$ is an $R$-algebra, then $F_A(M \underset{R}{\otimes} A) = F_R(M)A$.
\end{enumerate}

More generally, suppose that $\mathcal{F}$ is a coherent $\O_S$-module for some scheme $S$. Then the $h^{th}$ Fitting ideal $F^h(\mathcal{F})$ is defined to be the ideal sheaf of $S$ defined locally on affine subvarieties of $S$ as the Fitting ideal of the corresponding finitely presented module. In the language of coherent sheaves, properties $(1)$ through $(3)$ above are translated as follows.
\begin{enumerate}
\item $F^h(\mathcal{F})$ is independent of the local presentation for $\mathcal{F}$.
\item If $\mathcal{G}$ is also a coherent $\O_S$-module, then
$F(\mathcal{F} \oplus \mathcal{G}) = F(\mathcal{F})F(\mathcal{G})$.
\item If $T$ is an $S$-scheme given by $f:T \to S$, then
$F_T(f^{\ast}\mathcal{F}) = f^{-1}F_S(\mathcal{F})\cdot \O_T$.
\end{enumerate}

Given $\mathcal{F}$ as above, the \em Fitting rank \em of $\mathcal{F}$ is defined to be the largest integer $h$ such that $F^h(\mathcal{F}) = 0$. An $S$-point of $\mathcal{W}_d^r(X)$ is then defined to be a degree $d$ line bundle $\lambda \in \Bun_{\mathbb{G}_m,d}(X)(S)$ such that the Fitting rank of $R^1p_{\ast}\lambda$ is at least $g-d+r$. While this may appear quite different from the initial definition of the $\C$-points of $W_d^r(X)$, note that $\lambda \in \mathcal{W}_d^r(X)(S)$ implies that $h^1(X_s,\lambda_s) \geq g-d+r$ for each $s \in S$. By Riemann-Roch, this equivalent to $h^0(X_s,\lambda_s) \geq r+1$ for each $s \in S$.

An $S$-point of $\mathcal{G}_d^r(X)$ is defined to be a pair $(\lambda, V)$ where $\lambda \in \Bun_{\mathbb{G}_m,d}(X)(S)$ and $V$ is a locally free subsheaf of $p_{\ast}\lambda$ of rank $r$ which remains a subsheaf under arbitrary base change.

Rather than considering $\mathcal{G}_d^0(X)$, we will instead consider the moduli stack of pairs $(\lambda,s)$ where $\lambda$ is a line bundle on $X$ of degree $d$, $s \in H^0(X,\lambda)$, and $h^0(X,\lambda) \geq 1$.

Since the moduli of pairs $(\lambda,s)$ with $s \in H^0(X,\lambda)$ is a relative affine cone of $\mathcal{G}_d^0(X)$ over the fibers of $\Bun_{\mathbb{G}_m}(X)$, we will let $C\mathcal{G}(X)$ denote the moduli stack whose $S$-points consist of all pairs $(\lambda,s)$ in which $\lambda \in \mathcal{W}_d^0(X)(S)$ and $s \in H^0(X_S,\lambda)$.  $C\mathcal{G}_d(X)$ will denote the connected component of degree $d$.

We now investigate the extent to which the results of Theorem \ref{grd geometry} apply to the geometry of $\mathcal{W}_d^0(X)$ and $C\mathcal{G}_d(X)$.  The first obvious observation is that $W_d^r(X)$ and $G_d^r(X)$ are atlases for $\mathcal{W}_d^r(X)$ and $\mathcal{G}_d^r(X)$, respectively.  Therefore we only need to consider $C\mathcal{G}_d(X)$.

To begin, let us review the tangent space computations of $G^r_d(X)$ and $W^r_d(X)$ found in \cite[Sec. 4.4]{ACGH}. Let
\begin{equation*}
v:G^r_d(X) \to \Pic_d(X)
\end{equation*}
denote the projection map. Then, given $(\lambda,V) \in G^r_d(X)$,
we have the following exact sequence of tangent spaces
\begin{equation} \label{grd tangent}
0 \to T_{(\lambda,V)}(v^{-1}(\lambda)) \to T_{(\lambda,V)}(G^r_d(X))
\overset{v_{\ast}}{\to} T_{\lambda}(\Pic_d(X)).
\end{equation}

In \eqref{grd tangent}, we are interested in computing
$T_{(\lambda,V)}(G^r_d(X))$, which means that we need to compute
both $T_{(\lambda,V)}(v^{-1}(\lambda))$ and the image of $v_{\ast}$.

$T_{(\lambda,V)}(v^{-1}(\lambda))$ corresponds to first-order deformations of the pair $(\lambda,V)$ in which $\lambda$ is deformed trivially. Therefore, the first-order deformations coincide with the first-order deformations of $V$ as an element of the Grassmannian $Gr_{r+1}(H^0(X,\lambda))$. We conclude that \begin{equation*}
T_{(\lambda,V)}(v^{-1}(\lambda)) \simeq \Hom(V,H^0(X,\lambda)/V).
\end{equation*}

The image of $v_{\ast}$ corresponds to those first-order
deformations of $\lambda$ such that the subspace $V \subset
H^0(X,\lambda)$ may be deformed in a compatible way. Recall first
the well-known identification
\begin{equation*}
T_{\lambda}(\Pic_d(X)) \simeq H^1(X,\O_X).
\end{equation*}

To compute the image of $v_{\ast}$, it will be most convenient to
fix a section $s \in H^0(X,\lambda)$ and a class $\phi \in
H^1(X,\O_X)$, and to find necessary conditions so that there exists
a section $s_{\phi}$ of the deformation $\lambda_{\phi}$ of
$\lambda$ such that the restriction of $s_{\phi}$ to $X$ is $s$. We
may view $\lambda_{\phi}$ as sitting in a short exact sequence
\begin{equation}\label{line deform}
0 \to \lambda \to \lambda_{\phi} \to \lambda \to 0
\end{equation}
of $\O_X$-modules, in which the additional structure of
$\O_{X[\epsilon]}$-module comes from composing the projection
$\lambda_{\phi} \to \lambda$ with the inclusion $\lambda \to
\lambda_{\phi}$. Then \eqref{line deform} yields a long exact
sequence
\begin{equation*}
\ldots \to H^0(X,\lambda_{\phi}) \overset{r_{\phi}}{\to}
H^0(X,\lambda) \overset{\delta_{\phi}}{\to} H^1(X,\lambda) \to
\ldots,
\end{equation*}
in which compatible deformations of $s$ correspond to the inverse
image $r_{\phi}^{-1}(s)$. Therefore, $\phi \in H^1(X,\O_X)$ is in
the image of $v_{\ast}$ if and only if $s$ is in the image of
$r_{\phi}$. Since $\im(r_{\phi}) = \ker(\delta_{\phi})$ and
$\delta_{\phi}$ is given by
\begin{equation*}
\delta_{\phi}(t) = \phi \cup t,
\end{equation*}
the image of $v_{\ast}$ is
\begin{equation}\label{image}
\{ \phi \in H^1(X,\O_X): \phi \cup s = 0 \in H^1(X,\lambda) \text{
for all }s \in V \}.
\end{equation}

We can use the tangent space computation for $G^r_d(X)$ to also
compute the tangent spaces for $W^r_d(X)$. Since the restriction of
the projection
\begin{equation*}
G^r_d(X) \to W^r_d(X)
\end{equation*}
is an isomorphism away from $W^{r+1}_d(X)$, we see that
\begin{equation}\label{wrd tangent}
T_{\lambda}(W^r_d(X)) = \{\phi \in H^1(X,\O_X):\phi \cup s = 0
\text{ for all }s \in H^0(X,\lambda)\}, \text{ if }\lambda \notin
W^{r+1}_d(X).
\end{equation}
Furthermore, general results on determinantal varieties imply that
(see \cite[Sec. 2.2]{ACGH})
\begin{equation}\label{wr1d tangent}
T_{\lambda}(W^r_d(X)) = H^1(X,\O_X), \text{ if }\lambda \in
W^{r+1}_d(X).
\end{equation}

The above argument will need to be slightly altered to work for
$C\mathcal{G}_d(X)$. The first issue is that $G^r_d(X)$ is a scheme,
whereas $C\mathcal{G}_d(X)$ is a stack. We will therefore have to
replace tangent spaces with naive tangent complexes\footnote{We use
the term ``naive tangent complex" because it can be defined for any
sheaf of groupoids, and should not be confused with the tangent
complex of L. Illusie (\cite{Il}).}, and to replace the dimension of
the tangent space by the Euler characteristic of the naive tangent
complex. Then, just as an equidimensional scheme of dimension $d$ is
smooth at a point $x$ if and only if the dimension of the tangent
space at $x$ is $d$, an equidimensional stack of dimension $d$ is
smooth at a point $x$ if and only if the Euler characteristic of the
naive tangent complex at $x$ is $d$.

However, these differences are actually quite mild. If $F$ is a
sheaf of groupoids and $x$ is a $\C$-point of $F$, then the naive
tangent complex $\mathcal{T}_x(F)$ is a $2$-term complex lying in
degrees $-1$ and $0$ such that
\begin{equation*}
\begin{split}
H^0(\mathcal{T}_x(F))& = \{ \text{isomorphism classes of first-order
deformations of }x \},\\
H^{-1}(\mathcal{T}_x(F)) & = \text{Lie}(\Aut(x)).
\end{split}
\end{equation*}

We are now in a good position to verify the following geometric
properties of $C\mathcal{G}_d(X)$.

\begin{prop}\label{cgd}
$C\mathcal{G}_d(X)$ is irreducible of dimension $d$.
\begin{enumerate}
\item If $d<g$, a $\C$-point $(\lambda,s)$ of $C\mathcal{G}_d(X)$ is
smooth if and only if either $s \neq 0$, or $s=0$ and
$h^0(X,\lambda) = 1$.
\item If $g\leq d \leq 2g-2$, a $\C$-point
$(\lambda,s)$ of $C\mathcal{G}_d(X)$ is smooth if and only if either
$s \neq 0$, or $s=0$ and $H^1(X,\lambda) = 0$.
\item Finally, if $d > 2g-2$, then $C\mathcal{G}_d(X)$ is smooth.
\end{enumerate}
\end{prop}
\begin{proof}
Let $U$ denote the complement of the zero section inside
$C\mathcal{G}_d(X)$. Then the induced map $U \to \mathcal{G}^0_d(X)$
is smooth of relative dimension $1$. Since $\dim(G^0_d(X)) = d$, the
dimension of $\mathcal{G}^0_d(X)$ is $d-1$. Therefore
\begin{equation*}
\dim(U) = d.
\end{equation*}
Since $U$ is an irreducible (in fact, smooth) and dense open
substack of $C\mathcal{G}_d(X)$, it follows that $C\mathcal{G}_d(X)$
is irreducible of dimension $d$.

To determine the smooth points of $C\mathcal{G}_d(X)$, fix a
$\C$-point $(\lambda,s)$ of $C\mathcal{G}_d(X)$. The computation of
$H^0(\mathcal{T}_{(\lambda,s)}(C\mathcal{G}_d(X)))$ will be quite
similar to the tangent space computations for $G^r_d(X)$. In fact,
we can deduce from analogs of \eqref{grd tangent} and \eqref{image}
that a first-order deformation of $(\lambda,s)$ is given by the
following data.
\begin{enumerate}
\item A section $t \in H^0(X,\lambda)$.
\item A class $\phi \in H^1(X,\O_X)$ such that $\phi \cup s = 0$.
Furthermore, if $s=0$, we must have $\phi \in
T_{\lambda}(W^0_d(X))$.
\end{enumerate}

In order to find the isomorphism classes of first-order
deformations, note that an automorphism of $\lambda_{\phi}$ which
preserves $\lambda$ is given by some $1+a\epsilon \in \C[\epsilon]$.
Then an isomorphism between $(\lambda_{\phi},s+t_1\epsilon)$ and
$(\lambda_{\phi},s+t_2\epsilon)$ is given by an element
$1+a\epsilon$ such that
\begin{equation*}
(1+a\epsilon)(s+t_1\epsilon) = s+t_2\epsilon.
\end{equation*}
Therefore, $(\lambda_{\phi},s+t_1\epsilon) \simeq
(\lambda_{\phi},s+t_2\epsilon)$ if and only if
\begin{equation*}
t_1 + as = t_2
\end{equation*}
for some $a \in \C$. If $s \neq 0$, then a given first-order
deformation of $(\lambda,s)$ therefore lives in a $1$-dimensional
family of isomorphic deformations. On the other hand, if $s=0$, then
any first-order deformation of $(\lambda,s)$ has no non-trivial
automorphisms.

To completely determine
$H^0(\mathcal{T}_{(\lambda,s)}(C\mathcal{G}_d(X)))$, we now classify those $\phi \in H^1(X,\O_X)$ such that $\phi \cup s = 0$ and such
that $\phi \in T_{\lambda}(W^0_d(X))$. If $s \neq 0$, then the
computation is straightforward. In this case, $s$ determines a short
exact sequence
\begin{equation*}
0 \to \O_X \overset{s}{\to} \lambda \to T \to 0
\end{equation*}
in which the quotient $T$ is torsion. Therefore $H^1(X,T) = 0$, and
the induced long exact sequence ends with
\begin{equation*}
\begin{split}
H^1(X,\O_X)& \to H^1(X,\lambda) \to 0.\\
\alpha& \mapsto \alpha \cup s
\end{split}
\end{equation*}
Surjectivity shows that the dimension of $\{ \phi \in
H^1(X,\O_X):\phi \cup s = 0 \}$ is $g - h^1(X,\lambda)$.

If $s=0$, then \eqref{wrd tangent} and \eqref{wr1d tangent}
classify those $\phi \in H^1(X,\O_X)$ corresponding to first-order
deformations of $(\lambda,0)$. Let us now compute
$h^0(\mathcal{T}_{(\lambda,s)}(C\mathcal{G}_d(X)))$.
\begin{enumerate}
\item Suppose that $s \neq 0$. Then
\begin{equation*}
h^0(\mathcal{T}_{(\lambda,s)}(C\mathcal{G}_d(X))) =
h^0(X,\lambda)-1+g-h^1(X,\lambda) = d.
\end{equation*}
\item Suppose that $d<g$ and that $s=0$. Then
\begin{equation*}
\begin{split}
h^0(\mathcal{T}_{(\lambda,0)}(C\mathcal{G}_d(X)))& =
h^0(X,\lambda)+g-h^1(X,\lambda) = d+1, \text{ if }h^0(X,\lambda) =
1,\\
h^0(\mathcal{T}_{(\lambda,0)}(C\mathcal{G}_d(X)))& =
h^0(X,\lambda)+g, \text{ if }h^0(X,\lambda) > 1.
\end{split}
\end{equation*}
\item Suppose that $g\leq d\leq 2g-2$ and that $s = 0$. Then
\begin{equation*}
h^0(\mathcal{T}_{(\lambda,s)}(C\mathcal{G}_d(X))) =
h^0(X,\lambda)+g.
\end{equation*}
We note that $h^0(X,\lambda)+g = d+1$ if and only if $h^1(X,\lambda)
= 0$.
\end{enumerate}
The calculation of
$h^1(\mathcal{T}_{(\lambda,s)}(C\mathcal{G}_d(X)))$ is much simpler.
An automorphism of $(\lambda,s)$ is given by a scalar $a \in
\C^{\times}$ such that $as = s$. Therefore if $s \neq 0$, then
$(\lambda,s)$ has no non-trivial automorphisms, while if $s=0$, then
$(\lambda,0)$ has automorphism group $\C^{\times}$.

It is now clear that, for $d \leq 2g-2$,
\begin{equation*}
\chi(\mathcal{T}_{(\lambda,s)}(C\mathcal{G}_d(X)) = d
\end{equation*}
if and only if
\begin{enumerate}
\item $s \neq 0$, or
\item $s = 0$, $d<g$, and $h^0(X,\lambda) = 1$, or
\item $s = 0$, $g\leq d \leq 2g-2$, and $h^1(X,\lambda) = 0$.
\end{enumerate}
Lastly, to see that $C\mathcal{G}_d(X)$ is smooth whenever $d>2g-2$,
it is easiest to note that $H^1(X,\lambda) = 0$ in this case and
that $C\mathcal{G}_d(X)$ is therefore a vector bundle of rank
$d+1-g$ over $\Bun_{\mathbb{G}_m}(X)$.
\end{proof}

\section{A global Springer resolution}\label{s:resolution}

In this main section of the paper, we specialize to $G=SL_2$. Correspondingly, the associated Laumon/Drinfeld compactification will be denoted by $\overline{\Bun}_B(X)$. Furthermore, we will assume that $\deg(\L)$ is even and that $\deg(\L) \geq 2g$ (with the exception of Section \ref{ss:lower degree}). The same assumptions on the degree of $\L$ may also be found in \cite{Ngo} and \cite{Y}.

We begin by constructing a partial global analog of the Springer
resolution, denoted $\widehat{\N}$. It turns out that the geometry
of $\widehat{\N}$ is closely tied to the geometry of line bundles on
curves. For this reason, as the genus of the curve increases,
$\widehat{\N}$ becomes more complicated to understand. In fact,
$\widehat{\N}$ is only smooth (and hence an actual resolution) when
the genus of the curve is $0$ or $1$.  For this reason, after
showing that $\widehat{\N}$ is proper and birational over $\N$ in
Section \ref{ss:proper birational}, we will begin by describing
$\widehat{\N}$ when $X = \mathbb{P}^1$, where everything is as
simple as possible. This situation is illuminating because many of
the main ideas are presented without complications arising from the
particular geometry of the curve. After this, we will describe
$\widehat{\N}$ when $X$ is an elliptic curve.  This provides a
useful bridge between the genus $0$ case and the higher genus cases.
While we still obtain an honest resolution when $X$ is an elliptic
curve, some extra care must be given in regards to the particular
geometry of the curve. Finally, we will then study $\widehat{\N}$ in
the case of curves of genus greater than $1$. While $\widehat{\N}$
is not smooth in this case, its geometry is
intimately related to that of $C\mathcal{G}(X)$. We will then be
able to resolve $\widehat{\N}$ further, and obtain a stack
$\widetilde{\N}$ which is a resolution of singularities of $\N$.
Finally, we end by discussing twisting bundles $\L$ of lower degree,
as well as the stable locus of $\N$.

\subsection{Construction of a global Springer resolution}\label{ss:construction}

The purpose of this section is to provide definitions and basic
results which will be applicable regardless of the genus of $X$.

\begin{dfn}
Suppose that $(\lambda \subset E) \in \overline{\Bun}_B(X)(\C)$.
Then $E/\lambda \simeq F \oplus T$ where $F$ is a line bundle and $T$ is torsion.  The unique line bundle $\widetilde{\lambda} \supset \lambda$ such that $E/\widetilde{\lambda} \simeq F$ is called the \em normalization \em of $\lambda \subset E$.  The effective divisor on which the torsion sheaf $T \simeq \widetilde{\lambda}/\lambda$ is supported is called the \em defect \em of $\lambda \subset E$.  The defect will be denoted by $\df(\lambda \subset E)$, or by $\df(\lambda)$.
\end{dfn}

In order to define a candidate for the global Springer resolution, it is necessary to extend the definition of defect to arbitrary $S$-points of $\overline{\Bun}_B(X)$. That is, given an $S$-point $\lambda \subset E$ of $\overline{\Bun}_B(X)$, we seek a relative effective Cartier divisor $\df(\lambda)$ on $X_S$ which measures the failure of $\lambda$ to be a subbundle of $E$, and such that $\df(\lambda)_s = \df(\lambda_s)$ for all geometric points $s \in S$. Just as when $S = \Spec(\C)$, the divisor $\df(\lambda)$ will be referred to as the \em defect \em of $\lambda \subset E$.

The solution to this problem is provided by Fitting ideals, which were reviewed in Section \ref{ss:divisor geometry}. Before stating the result, there is one further basic property of Fitting ideals not previously mentioned. As before, both module and sheaf theoretic versions are given.
\begin{enumerate}
\item If $R$ is a discrete valuation ring with maximal ideal
$\mathfrak{m}$ and $M$ is a finitely generated $R$-module, then
$F^0(M) = \mathfrak{m}^{\ell(M)}$, where $\ell(M)$ denotes the
length of $M$ as an $R$-module.
\item If $X$ is a smooth curve and there is an exact sequence
\begin{equation*}
0 \to \lambda \to E \to Q \to 0
\end{equation*}
with $\lambda$ a line bundle and $E$ a vector bundle of rank $r$,
then $F^{r-1}(Q)$ is the ideal sheaf of the divisor on which the
torsion part of $Q$ is supported. In other words, $F^{r-1}(Q)$ is
the ideal sheaf of the defect of $\lambda \subset E$.
\end{enumerate}

The following proposition shows that if $(\lambda \subset E) \in \overline{\Bun}_B(X)(S)$, then the first Fitting ideal of the quotient is the sought after generalization of the defect to $S$-points of $\overline{\Bun}_B(X)$.

\begin{prop}\label{defect S}
Given a scheme $S$, suppose that
\begin{equation*}
0 \to \lambda \to E \to Q \to 0
\end{equation*}
is an exact sequence of coherent sheaves on $X_S$ where $\lambda$ is a line bundle and $E$ is a vector bundle of rank $r$, and such that $Q$ is $S$-flat. Then there exists a relative effect Cartier divisor $\df(\lambda)$ on $X_S$ over $S$ which measures the failure of $\lambda$ to be a subbundle of $E$, and such that the fiber $\df(\lambda)_s$ coincides with the defect of the fiber $\df(\lambda_s)$ for every geometric point $s \in S$.
\end{prop}
\begin{proof}
We claim that we may define $\df(\lambda)$ to be the closed subscheme of $X_S$ defined by the ideal sheaf $F^{r-1}(Q)$. To show that $\df(\lambda)$ is a relative Cartier divisor on $X_S$, it suffices to prove, using \cite[Lem. 1.2.3]{KM}, that $\df(\lambda)$ is finite and flat over $S$. The closed subscheme $\df(\lambda)$ is proper over $S$ because $X_S$ is proper over $S$. Furthermore, $\df(\lambda)$ is quasi-finite over $S$ because $\df(\lambda)_s = \df(\lambda_s)$ for every $s \in S$. It follows that $\df(\lambda)$ is finite over $S$.

To show that $\df(\lambda)$ is $S$-flat, we use the local criterion for flatness (\cite[Thm. 6.8]{Eis}) and the locally free resolution of $Q$ given by
\begin{equation*}
0 \to \lambda \to E \to Q \to 0.
\end{equation*}
Since flatness is an open property, it suffices to check the local criterion of flatness only at closed points of $S$. Let $s \in S$ be a closed point, and pick $x \in X_S$ such that $p(x) = s$. Then we need to check that
\begin{equation}\label{tor}
\Tor_1^{i_s^{\ast}\O_S}(\O_{S,s},i_x^{\ast}\O_{\df(\lambda)}) = 0.
\end{equation}
Then consider the corresponding map on the flat resolution of $Q$,
\begin{equation*}
f_s:i_x^{\ast}\lambda \otimes \O_{S,s} \to i_x^{\ast}E \otimes \O_{S,s}.
\end{equation*}
The map $f_s$ is injective due to the assumption that $Q$ is $S$-flat. Furthermore, the ideal sheaf of $\df(\lambda)$ at $s$ is defined by the ideal generated by the entries of $f_s$. Since $X$ is a smooth curve, the entries of $f_s$ are elements of a discrete valuation ring. Therefore, since $f_s$ is injective, the ideal generated by the entries of $f_s$ is a principal ideal generated by some $t^k$ where $t$ is a uniformizing parameter and $k \geq 0$. Therefore the desired equality in \eqref{tor} holds.

We have thus shown that $\df(\lambda)$ is a relative effective Cartier divisor on $X_S$ over $S$ which coincides with $\df(\lambda_s)$ over the geometric points of $S$. Finally, the fiber-by-fiber criterion for flatness over $X_S$ (\cite[Cor V.3.6]{AK}) implies that $\df(\lambda)$ is the locus where $Q$ is not locally free (equivalently, where $\lambda$ fails to be a subbundle of $E$).
\end{proof}

We are now ready to define the partial global Springer resolution.

\begin{dfn}\label{global springer}
The \em partial global Springer resolution \em $\widehat{\N}_{X,G,\L}$ (or simply $\widehat{\N}$) is the moduli stack whose $S$-points consist of
\begin{equation*}
((E,\varphi),\lambda \subset E) \in \N(S)
\underset{\Bun_G(X)(S)}{\times} \overline{\Bun}_{B,d}(X)(S)
\end{equation*}
such that the following conditions hold.
\begin{enumerate}
\item $\lambda \subset \ker(\varphi)$.
\item $\im(\varphi) \subset (\lambda \otimes \L)(-\df(\lambda))$.
\item $\lambda^{\otimes 2} \otimes \pi^{\ast}\L \in \mathcal{W}^0(X)(S)$.
\end{enumerate}
\end{dfn}

The definition of $\widehat{\N}$ deserves some explanation. Condition (1) is motivated by the ordinary Springer resolution, which consists of flags preserved by a nilpotent endomorphism. It also causes $\widehat{\N}$ to be a closed substack of $T^{\ast}\overline{\Bun_B}(X)$ when $\L = \omega_X$. Condition (2) can be motivated in a couple of ways. Since $\varphi$ is nilpotent, $\im(\varphi) \subset \ker(\varphi) \otimes \L$, and condition (2) can be viewed as a strengthening of this. Indeed, in the ideal situation where $\varphi$ is generically regular and $\lambda = \ker(\varphi)$, the defect is zero and nothing new occurs. In this case, we are then largely able to understand $(E,\varphi)$ through the data of $\lambda$ and the induced map $\overline{\varphi}:\im(\varphi) \to \lambda \otimes \L$. Since $\lambda$ is a subbundle, $\im(\varphi) = \lambda^{-1}$, and $\overline{\varphi}$ can be viewed as a global section of $\lambda^{\otimes 2} \otimes \L$.

In a setting where $\lambda \subset E$ is not a subbundle, we would still like to have the data of $\overline{\varphi} \in H^0(X,\lambda^{\otimes 2} \otimes \L)$, and this is precisely what condition (2) provides. This is desirable for a couple of reasons. The first is that without this condition, $\widehat{\N}$ would not be birationally equivalent to $\N$, as most of the fibers would be much too large. Secondly, condition (2) makes it possible for $\widehat{\N}$ to be a vector bundle over $\overline{\Bun_B}(X)$, much like $\widetilde{\N \mathrm{ilp}}$ is a vector bundle over $G/B$. For certain degrees of $\lambda$, this will not be true due to the non-constancy of $h^0(X,\lambda^{\otimes 2} \otimes \L)$, but even in this situation condition (2) still allows us to understand $\widehat{\N}$ in terms of $C\mathcal{G}(X)$.

There is also a motivation for condition (2) coming from the notion of the irregularity of the nilpotent Higgs field $\varphi$. Intuitively speaking, the irregularity of $\varphi$, denoted $\irr(\varphi)$, is the divisor which measures the failure of a generically regular $\varphi$ to be regular. Since $G = SL_2$, regularity of $\varphi$ is equivalent to $\varphi$ having rank $1$ everywhere, which means that $\irr(\varphi)$ is roughly the divisor of zeroes of $\varphi$. More precisely, if $\varphi$ is generically regular, $\irr(\varphi) = \df(\im(\varphi) \subset E \otimes \L)$. Then condition (2) can be viewed as saying that $2\df(\lambda) \subset \irr(\varphi)$ for generically regular $\varphi$. This simply says that there should be a relationship between the failure of $\lambda$ to be a subbundle of $E$ and the failure of $\varphi$ to be regular.

\begin{exmp}
Let $X = \mathbb{P}^1$ and $\L = \O(2)$.  Consider the nilpotent endomorphism $\varphi:\O \oplus \O \to (\O \oplus \O) \otimes \O(2)$ given by $\varphi = \begin{pmatrix}0 & z^2 \\ 0 & 0 \end{pmatrix}$. It turns out that $(\O \oplus \O, \varphi)$ lies in the intersection of two irreducible components of $\N(\C)$. Without condition (2) of Definition \ref{global springer}, the fiber of $(\O \oplus \O, \varphi)$ in $\widehat{\N}_1$ would correspond to all inclusions $\O(-1) \subset \O \oplus \O$ of the form $\begin{pmatrix}s \\ 0 \end{pmatrix}$ with $s \in H^0(\O(1))$, and hence would be isomorphic to $\mathbb{P}^1$.  Furthermore, this example generalizes to any $(\O \oplus \O, \varphi)$ such that $\irr(\varphi)$ is a skyscraper sheaf supported on $2x$ for some $x \in X$.  However, condition (2) of Definition \ref{global springer} implies that the section $s \in H^0(\O(1))$ must be of the form $\begin{pmatrix}z \\ 0 \end{pmatrix}$, and hence the fiber consists of a single point in this connected component.
\end{exmp}

Finally, condition (3) of Definition \ref{global springer} is there for a couple of reasons. First, it dictates that $\deg(\lambda) \geq -\frac{1}{2}\deg(\L)$. Therefore, the connected components of $\widehat{N}$ are indexed by:
\begin{enumerate}
\item All integers $d > -\frac{1}{2}\deg(\L)$.
\item All $2^{2g}$ square roots of $\L^{-1}$.
\end{enumerate}

For each such integer $d$, the corresponding connected component will be denoted $\widehat{N}_d$. For convenience, this notation will also be used when $d = -\frac{1}{2}\deg(\L)$ to denote the union of connected components for which $\deg(\lambda) = -\frac{1}{2}\deg(\L)$. In other words,
\begin{equation*}
\widehat{\N}_{-\frac{1}{2}\deg(\L)} = \underset{L}{\bigcup} \widehat{N}_L,
\end{equation*}
where $L^{\otimes 2} = \L^{-1}$.

Lastly, requiring that $\lambda^{\otimes 2} \otimes \L \in \mathcal{W}^0(X)$ prevents superfluous Higgs bundles from appearing in the fibers of the projection map $\widehat{\N} \to \overline{\Bun_B}(X)$ in the sense that any fiber will necessarily contain nilpotent Higgs bundles with a nonzero Higgs field.

\subsection{Properness and birationality of $\widehat{\N} \to
\N$}\label{ss:proper birational}

The purpose of this section is to show that the projection
\begin{equation*}
\mu:\widehat{\N} \to \N
\end{equation*}
is both proper and a birational equivalence.  Issues of smoothness and irreducibility of $\widehat{\N}$ will be addressed in subsequent sections.

\begin{prop}\label{proper}
$\mu$ is proper.
\end{prop}
\begin{proof}
Consider the following diagram.
\begin{diagram}
\overline{\Bun}_B(X) & \lTo & \overline{\Bun}_B(X) \underset{\Bun_G(X)}{\times} \N \\
\dTo && \dTo \\
\Bun_G(X) & \lTo & \N \\
\end{diagram}

Since $\overline{\Bun}_B(X) \to \Bun_G(X)$ is proper and $\N$ is a
closed substack of $\mathcal{M}$, the pull-back map
\begin{equation*}
\overline{\Bun}_B(X) \underset{\Bun_G(X)}{\times} \N \to \N
\end{equation*}
is also proper.  Therefore in order to show that $\widehat{\N} \to
\N$ is proper, it suffices to note that $\widehat{\N}$ is a closed
substack of $\overline{\Bun}_B(X) \underset{\Bun_G(X)}{\times} \N$.
\end{proof}

To show that $\mu$ is a birational equivalence between
$\widehat{\N}$ and $\N$, we will define an open substack of
$\N$ over which $\mu$ is an isomorphism, which will be referred to as the locus of \em globally regular \em nilpotent Higgs bundles. To define this substack, first recall the algebraic stack $Coh_{X,0}$ as defined in \cite{Lau}. $Coh_{X,0}$ is the moduli stack of finite length coherent sheaves on $X$. Its connected components are given by $Coh_{X,0}^m$, which denotes the moduli stack of coherent sheaves on $X$ of length $m$.

For a fixed $m$, each $Coh_{X,0}^m$ is stratified by locally closed substacks $Coh_{X,0}^{(m_1,\ldots,m_k)}$, where $(m_1,\ldots,m_k)$ is a partition of $m$. The unique open stratum corresponds to the the partition $(m)$. In this case, $Coh_{X,0}^{(m)}$ parameterizes those length $m$ coherent sheaves on $X$ which are supported on a divisor of the form $\underset{i=1}{\overset{m}{\sum}}x_i$, in which the $x_i$ are distinct points of $X$.

Let $\N^{\mathrm{gen,reg}}$ denote the open locus of generically regular elements of $\N$. Then there is a map
\begin{equation*}
\alpha:\N^{\mathrm{gen,reg}} \to Coh_{X,0}
\end{equation*}
which is defined as follows. Recalling that for a generically regular $(E,\varphi)$, the irregularity $\irr(\varphi)$ is the divisor which is defined to be the defect of $\im(\varphi) \subset E \otimes \L$, the map $\alpha$ is defined by
\begin{equation*}
\alpha((E,\varphi)) = \O_{\irr(\varphi)}.
\end{equation*}
Finally, the globally regular substack $\N^{\mathrm{gl,reg}}$ of $\N$ is defined to be
\begin{equation*}
\alpha^{-1}\left(\underset{m \geq 0}{\bigcup} Coh_{X,0}^{(m)}\right).
\end{equation*}

$\N^{\mathrm{gl,reg}}$ is therefore an open substack of $\N^{\mathrm{gen,reg}}$ because each $Coh_{X,0}^{(m)}$ is open in $Coh_{X,0}^m$. Since $\N^{\mathrm{gen,reg}}$ is an open substack of $\N$, it follows that $\N^{\mathrm{gl,reg}}$ is as well.

\begin{prop}\label{birational}
$\mu:\widehat{\N} \to \N$ is a birational equivalence.  More specifically, the restriction of $\mu$ to $\N^{\mathrm{gl,reg}}$ is an isomorphism.  The inverse map is given by
\begin{equation*}
\begin{split}
\nu:\N^{\mathrm{gl,reg}}& \to \widehat{\N}\\
(E,\varphi)& \mapsto (\ker(\varphi)\subset E, \varphi).
\end{split}
\end{equation*}
\end{prop}
\begin{proof}
It is clear that $\mu \circ \nu = \mathrm{id}_{\N^{\mathrm{gl,reg}}}$. To show that $\nu \circ \mu = \mathrm{id}_{\widehat{\N}}$, it suffices to show that the fiber over a point $(E,\varphi) \in \N^{\mathrm{gl,reg}}$ consists of a single point. Suppose then that $\lambda \subset E$ is in the fiber over $(E,\varphi)$. Since $2\df(\lambda) \subset \irr(E,\varphi)$, we must have $\df(\lambda) = 0$, which implies that $\lambda = \ker(\varphi)$.
\end{proof}

\subsection{A global Springer resolution in genus 0}\label{ss:genus
0}

In this section, $X = \mathbb{P}^1$ and $\L$ will be any line bundle on $X$ of even, nonnegative degree.

Let us briefly explain what aspect of $\mathbb{P}^1$ makes its global nilpotent cone so much easier to study than that of higher genus curves.  Recall that on an arbitrary curve, the Euler characteristic of any line bundle is a linear function of the degree of the line bundle.  On the other hand, for any line bundle on $\mathbb{P}^1$ of nonnegative degree, $h^0$ is a linear function of the degree of the line bundle (i.e., $H^1$ vanishes). This is important because for any $(\lambda \subset E) \in \overline{\Bun}_B(X)$ in the image of the projection from $\widehat{\N}$, the geometry of $\widehat{\N}$ will largely be controlled by $H^0(X,\lambda^{\otimes 2} \otimes \L)$.  By condition (3) of Definition \ref{global springer}, the degree of $\lambda^{\otimes 2} \otimes \L$ is nonnegative, and so the dimension of $H^0(X,\lambda^{\otimes 2} \otimes \L)$ depends only on the degree of $\lambda$ when $X = \mathbb{P}^1$.

\begin{thm}\label{resolution 0}
When $X = \mathbb{P}^1$, the projection map $\mu:\widehat{\N} \to \N$ is a resolution of singularities.
\end{thm}
\begin{proof}
Given Proposition \ref{proper} and Proposition \ref{birational}, all that is left to prove is that $\widehat{\N}$ is smooth.  This follows from Proposition \ref{smooth 0} below together with the smoothness of $\overline{\Bun}_B(X)$.
\end{proof}

\begin{prop}\label{smooth 0}
The projection map $\widehat{\N}_d \to \overline{\Bun}_{B,d}(X)$ is
a vector bundle of rank $2d + \deg(\L) + 1 $.
\end{prop}
\begin{proof}
For any $S$-point of $\overline{\Bun}_{B,d}(X)$, consider the
following pull-back diagram.
\begin{diagram}
\widehat{\N}_{S,d}:= S \underset{\overline{\Bun}_{B,d}(X)}{\times}
\widehat{\N}_d & \rTo & \widehat{\N}_d \\
\dTo && \dTo \\
S & \rTo & \overline{\Bun}_{B,d}(X)
\end{diagram}

If $S \to \overline{\Bun}_{B,d}(X)$ corresponds to $\lambda \subset
E$, then $\widehat{\N}_{S,d}$ is the total space of the sheaf
$p_{\ast}(\lambda^{\otimes 2} \otimes \pi^{\ast}\L)$, which is
coherent because $p$ is proper.  To show that
$p_{\ast}(\lambda^{\otimes 2} \otimes \pi^{\ast}\L)$ is locally
free, it therefore suffices to show that its fibers have constant
dimension, provided that $S$ is reduced.  Since
$\deg(\lambda^{\otimes 2} \otimes \pi^{\ast}\L) \geq 0$ and $X = \mathbb{P}^1$,
\begin{equation*}
h^0(X_s,\lambda^{\otimes 2}_s \otimes \L) = \chi(X_s, \lambda^{\otimes 2}_s \otimes \L) = 2d + \deg(\L) + 1
\end{equation*}
for every $s \in S$. Finally, to show that $\widehat{\N}_d \to
\overline{\Bun}_{B,d}(X)$ is a vector bundle, it suffices to show
that $\widehat{\N}_{S,d} \to S$ is a vector bundle when $S$ is an
atlas for $\overline{\Bun}_{B,d}(X)$.  Since
$\overline{\Bun}_{B,d}(X)$ is reduced, so is its atlas.
\end{proof}

Due to Theorem \ref{resolution 0}, we will refer to $\widehat{\N}$
as a \em global Springer resolution \em when $X$ is rational.

\begin{cor}\label{dim genus 0}
Each connected component of $\widehat{\N}$ has dimension $\deg(\L) - 1$.
\end{cor}
\begin{proof}
By Proposition \ref{smooth 0} and Proposition \ref{dimension B}, the dimension of $\widehat{\N}_d$ is given by
\begin{equation*}
(2d + \deg(\L) + 1) + (-2d - 2) = \deg(\L) - 1.
\end{equation*}
\end{proof}

\begin{cor}\label{equidim 0}
$\N$ is equidimensional of dimension $\deg(\L) - 1$.
\end{cor}
\begin{proof}
Since $\widehat{\N}$ is a resolution of singularities of $\N$ by
Theorem \ref{resolution 0}, the irreducible components $\N_d$ of
$\N$ are in bijection with the connected components $\widehat{\N}_d$
of $\widehat{\N}$.  Furthermore, the birational equivalence between
$\widehat{\N}_d$ and $\N_d$ for each $d$ implies that their
dimensions are the same.  Then Corollary \ref{dim genus 0} implies
that $\N$ is equidimensional of dimension $\deg(\L) - 1$.
\end{proof}

\subsection{A global Springer resolution in genus 1}\label{ss:genus
1}

$X$ denotes an elliptic curve in this section, and $\L$ is allowed to be any line bundle on $X$ of even degree at least $2$. Showing that $\widehat{\N}$ is a resolution of $\N$ for an elliptic
curve only requires slightly more care than it did for the
projective line. Indeed, if $\lambda$ is a line bundle on X of
nonnegative degree, then it is almost true that $h^0(\lambda)$ is a
linear function of the degree of $\lambda$. The only exception is
when $\deg(\lambda) = 0$, in which case $h^0(\lambda) = 0 \text{ or
}1$, depending on whether or not $\lambda = \O_X$.

\begin{thm}\label{resolution 1}
When $X$ is an elliptic curve, the projection map $\mu:\widehat{\N}
\to \N$ is a resolution of singularities.
\end{thm}
\begin{proof}
$\mu$ is a proper, birational equivalence by Proposition
\ref{proper} and Proposition \ref{birational}. The smoothness of
$\widehat{\N}_d$ for $d > -\frac{1}{2}\deg(\L)$ follows from
Proposition \ref{smooth 1} and the smoothness of
$\overline{\Bun}_B(X)$. When $d = -\frac{1}{2}\deg(\L)$, we are only
considering those $(\lambda \subset E) \in \overline{\Bun}_{B,d}(X)$
such that $h^0(\lambda^{\otimes 2} \otimes \L) \geq 1$. Then let
\begin{equation*}
\overline{r}_{d,\L}:\overline{\Bun}_{B,d}(X) \to
\Bun_{\mathbb{G}_m,2d+\deg(\L)}(X)
\end{equation*}
denote the map given by
\begin{equation*}
(\lambda \subset E) \mapsto \lambda^{\otimes 2} \otimes \L.
\end{equation*}
Just as we for curves of genus at least $2$, we can consider the substack $\mathcal{W}^0_0(X) \subset \Bun_{\mathbb{G}_m,0}(X)$ consisting of degree $0$ line bundles which possess a nonzero global section. Setting
\begin{equation*}
\overline{\Bun}_{B,d}^0(X) :=
(\overline{r}_{d,\L})^{-1}(\mathcal{W}^0_0(X)),
\end{equation*}
the smoothness of $\widehat{\N}_d$ follows from Proposition
\ref{smooth 1}, the smoothness of $\overline{r}_{d,\L}$, and the
smoothness of $\mathcal{W}^0_0(X) \simeq B\mathbb{G}_m$.
\end{proof}

\begin{prop}\label{smooth 1}
First suppose that $d > -\frac{1}{2}\deg(\L)$.  Then the projection
map $\widehat{\N}_d \to \overline{\Bun}_{B,d}(X)$ is a vector bundle
of rank $2d + \deg(\L)$. When $d = -\frac{1}{2}\deg(\L)$,
$\widehat{\N}_d \to \overline{\Bun}_{B,d}^0(X)$ is a line bundle.
\end{prop}
\begin{proof}
The proof is identical to that of Proposition \ref{smooth 0}.
\end{proof}

Due to Theorem \ref{resolution 1}, we will refer to $\widehat{\N}$
as a \em global Springer resolution \em when $X$ is an elliptic
curve.

As in the case of genus $0$, we can now extract corollaries about the equidimensionality of $\widehat{\N}$ and $\N$.

\begin{cor}
Each connected component of $\widehat{\N}$ has dimension $\deg(\L)$. Moreover, $\N$ is equidimensional of dimension $\deg(\L)$.
\end{cor}
\begin{proof}
The only part of the proof not identical to Corollary \ref{dim genus 0} and Corollary \ref{equidim 0} regards the dimension of the four connected components of $\widehat{\N}_{-\frac{1}{2}\deg(\L)}$, in which case it is necessary to compute the dimension of $\overline{\Bun}_{B,-\frac{1}{2}\deg(\L)}^0(X)$. Setting $d = -\frac{1}{2}\deg(\L)$, the map
\begin{equation*}
\overline{r}_{d,\L}:\overline{\Bun}_{B,d}(X) \to
\Bun_{\mathbb{G}_m,2d+\deg(\L)}(X)
\end{equation*}
factors as
\begin{equation}\label{factors}
\begin{diagram}
\overline{\Bun}_{B,d}(X)& \rTo & \Bun_{\mathbb{G}_m,d}(X) & \rTo &
\Bun_{\mathbb{G}_m,2d+\deg(\L)}(X)\\
(\lambda \subset E)& \rMapsto & \lambda & \rMapsto &
\lambda^{\otimes 2}\otimes \L
\end{diagram}
\end{equation}
The second map in \eqref{factors} is \'{e}tale, and the first map is
smooth by Lemma \ref{smooth lemma}, of the same relative dimension
as the corresponding projection map
\begin{equation*}
r_d:\Bun_{B,d}(X) \to \Bun_{\mathbb{G}_m,d}(X).
\end{equation*}
Given a line bundle $\lambda$ of degree $d$ on $X$,
\begin{equation*}
r_d^{-1}(\lambda) =
{\Ext}^1(\lambda^{-1},\lambda)\big/\Hom(\lambda^{-1},\lambda),
\end{equation*}
where $\Hom(\lambda^{-1},\lambda)$ acts trivially on $\Ext^1(\lambda^{-1},\lambda)$. Therefore Riemann-Roch implies that the relative dimension of $r_d$ is
\begin{equation*}
-2d = \deg(\L).
\end{equation*}
Finally, since $\mathcal{W}_0^0(X)$ has codimension $1$ in
$\Bun_{\mathbb{G}_m,0}(X)$, we find that
\begin{equation*}
\dim(\overline{\Bun}_{B,-\frac{1}{2}\deg(\L)}^0(X)) = \deg(\L) - 1.
\end{equation*}
Finally, Proposition \ref{smooth 1} implies that the corresponding
connected component of $\widehat{\N}$ has dimension $\deg(\L)$.
\end{proof}

\subsection{A higher genus global Springer resolution}\label{ss:higher
genus}

In this section, $X$ denotes a smooth, connected projective curve of genus $g \geq 2$, and $\L$ is a line bundle on $X$ of even degree at least $2g$. Since the first cohomology of a line bundle on $X$ is guaranteed to vanish if and only if the degree of the line bundle is at least $2g$, the methods used to study $\widehat{\N}$ when $g = 0 \text{ or }1$ will be insufficient for arbitrary $g$. In order to effectively analyze $\widehat{\N}$, we begin this section by relating the geometry of $\widehat{\N}$ to that of $C\mathcal{G}(X)$. The following lemma is a generalization of the well-known and easily proven fact that $\Bun_B(X) \to \Bun_T(X)$ is smooth.

\begin{lem}\label{smooth lemma}
The projection map $\overline{r}:\overline{\Bun}_B(X) \to
\Bun_{\mathbb{G}_m}(X)$ is smooth.
\end{lem}
\begin{proof}
Since $\overline{\Bun}_B(X)$ and $\Bun_{\mathbb{G}_m}(X)$ are smooth stacks, it suffices to show that $\overline{r}$ satisfies the formal smoothness criteria for $R = \C[\epsilon]$ and $I = (\epsilon)$. In other words, given a $\C$-point $f:\lambda \hookrightarrow E$ of $\overline{\Bun}_B(X)$ and a first-order deformation $\phi \in H^1(X,\O_X)$ of $\lambda$, it suffices to find a first-order deformation $\theta \in H^1(X,\mathrm{ad}(E))$ of $E$ and a first-order deformation $\tilde{f}:\lambda \hookrightarrow E$ of $f$ such that the triple $(\phi,\theta,\tilde{f})$ forms a first-order deformation of $f:\lambda \hookrightarrow E$.

Let $g_{ij}$ denote the transition $1$-cocycle for $\lambda$ and let $h_{ij}$ denote the transition $1$-cocycle for $E$, relative to a common trivializing open cover $\{ U_i \}$ for $\lambda$ and $E$. Then the condition that $f$ is an injection from $\lambda$ to $E$ means that $f \neq 0$ and that
\begin{equation*}
f_ig_{ij} = h_{ij}f_j,
\end{equation*}
along with the usual cocycle condition.

Given a fixed $\phi \in H^1(X,\O_X)$, a new transition
$1$-cocycle for the corresponding deformation of $\lambda$ is given by
\begin{equation*}
\tilde{g}_{ij} = (g_{ij} + \phi_{ij}\epsilon).
\end{equation*}
Similarly, given a choice of $\theta \in H^1(X,\mathrm{ad}(E))$, a transition $1$-cocycle for the corresponding deformation of $E$ is given by
\begin{equation*}
\tilde{h}_{ij} = (h_{ij} + \theta_{ij}\epsilon).
\end{equation*}
Finally, choices of $\tilde{f}_i:\lambda|_{U_i} \hookrightarrow E|_{U_i}$ give a first-order deformation of $f:\lambda \hookrightarrow E$, along with $\phi$ and $\theta$, if and only if
\begin{equation}\label{deformation}
(f_i + \tilde{f}_i\epsilon)\tilde{g}_{ij} = \tilde{h}_{ij}(f_j +
\tilde{f}_j\epsilon).
\end{equation}
Expanding \eqref{deformation}, we equivalently require that
\begin{align}
f_ig_{ij}& = h_{ij}f_j,\\
f_i\phi_{ij} - \theta_{ij}f_j& = h_{ij}\tilde{f}_j -
\tilde{f}_ig_{ij}.
\end{align}
The first equation is simply a restatement of the fact that $f:\lambda \to E$. In the second equation, the right-hand side is a coboundary, which means that $f_i\phi_{ij} - \theta_{ij}f_j=0$ as a cohomology class. Thus it suffices to find $\theta$ such that
\begin{equation*}
\theta f = \phi f \in H^1(X, E \otimes \lambda^{-1}).
\end{equation*}

Writing $f_j = \begin{pmatrix}s_j\\ t_j\end{pmatrix}$, we may then choose
\begin{equation*}
\theta_{ij} = \begin{pmatrix}(s_jt_j + 1)\phi_{ij} & -s_j^2\phi_{ij}\\ (t_j^2 + 2\frac{t_j}{s_j})\phi_{ij} & -(s_jt_j + 1)\phi_{ij}\end{pmatrix},
\end{equation*}
where all occurrences of $s_j$ and $t_j$ denote their restrictions to $U_{ij}$. The fact that this choice of $\theta_{ij}$ does define a $1$-cocycle with values in $\mathfrak{sl}_2$ follows from the fact that each of its entries are $1$-cocycles with values in $\C$.
\end{proof}

To relate $\widehat{\N}$ to $C\mathcal{G}(X)$, consider a variant of $C\mathcal{G}(X)$. Namely, let $C\mathcal{G}(X,\L)$ denote the moduli stack whose $S$-points consist of all pairs $(\lambda,s)$ where $\lambda^{\otimes 2} \otimes \pi^{\ast}\L \in \mathcal{W}^0_d(X)(S)$, and $s \in H^0(X_S,\lambda^{\otimes 2} \otimes \pi^{\ast}\L)$. Then $C\mathcal{G}(X,\L)$ fits into the following pull-back diagram.
\begin{equation}\label{cgdl}
\begin{diagram}
C\mathcal{G}(X,\L) & \rTo & C\mathcal{G}(X)\\
\dTo && \dTo\\
\Bun_{\mathbb{G}_m}(X) & \rTo^{\mathrm{Sq}_{\L}} & \Bun_{\mathbb{G}_m}(X)
\end{diagram}
\end{equation}

$\mathrm{Sq}_{\L}$ is the map which sends an $S$-point $\lambda$ to $\lambda^{\otimes 2} \otimes \pi^{\ast}\L$. Since the squaring map is \'{e}tale and the map which tensors with a fixed line bundle is an isomorphism, the map $C\mathcal{G}(X,\L) \to C\mathcal{G}(X)$ is \'{e}tale, which induces an \'{e}tale surjection (using the fact that $\deg(\L)$ is even)
\begin{equation*}
C\mathcal{G}_d(X,\L) \to C\mathcal{G}_{2d+\deg(\L)}(X)
\end{equation*}
from the degree $d$ connected component to the degree $2d+\deg(\L)$ connected component for each integer $d$.  Therefore, geometric properties of $C\mathcal{G}_d(X,\L)$ such as smoothness and irreducibility follow from the corresponding properties of $C\mathcal{G}_{2d+\deg(\L)}(X)$. Thus we have the following corollary of Proposition \ref{cgd}.

\begin{cor}
$C\mathcal{G}_d(X,\L)$ is irreducible of dimension $2d+\deg(\L)$. If $2d+\deg(\L) \leq 2g-2$, the singular locus of $C\mathcal{G}_d(X,\L)$ consists of pairs $(\lambda,0)$ where
\begin{enumerate}
\item  $\lambda^{\otimes 2} \otimes \L \in \mathcal{W}^1_{2d+deg(\L)}(X)$ if $2d+\deg(\L) < g$.
\item  $(\lambda^{\otimes 2} \otimes \L)^{-1} \otimes \omega_X \in \mathcal{W}^0_{2g-2-2d-deg(\L)}(X)$ if $g \leq 2d+\deg(\L) \leq 2g-2$.
\end{enumerate}
If $2d+\deg(\L) > 2g-2$, then $C\mathcal{G}_d(X,\L)$ is smooth.
\end{cor}

To bring $\widehat{\N}$ into the discussion, diagram \eqref{cgdl} can be extended to the following diagram in which each of the two small squares are pull-backs.
\begin{diagram}
\widehat{\N}_d & \rTo^{\widetilde{r}_d} & C\mathcal{G}_d(X,\L) & \rTo & C\mathcal{G}_{2d+\deg(\L)}(X)\\
\dTo && \dTo && \dTo\\
\overline{\Bun}_{B,d}(X) & \rTo^{\overline{r}_d} &
\Bun_{\mathbb{G}_m,d}(X) & \rTo^{\mathrm{Sq}_{\L}} &
\Bun_{\mathbb{G}_m,2d+\deg(\L)}(X)
\end{diagram}

The map $\widetilde{r}_d$ sends a point $(\lambda \subset E,
\varphi)$ to the pair $(\lambda,\overline{\varphi})$ where
$\overline{\varphi}$ is the section of $\lambda^{\otimes 2} \otimes
\L$ corresponding to the induced map
$\overline{\varphi}:\lambda^{-1} \to \lambda \otimes \L$.

By Lemma \ref{smooth lemma}, it follows that $\widetilde{r}_d$ is a smooth map. Therefore questions of smoothness and irreducibility of $\widehat{\N}$ have been reduced to issues of smoothness and irreducibility of $C\mathcal{G}(X,\L)$, and hence of $C\mathcal{G}(X)$. So, to completely resolve $\N$, it suffices to solve the simpler problem of resolving $C\mathcal{G}_d(X)$ for every $d \leq 2g-2$.

\begin{exmp}
There are some special examples where $C\mathcal{G}_d(X)$ is smooth even when $d \leq 2g-2$. The first such example occurs when $g=3$, $d=2$, and $X$ is not hyperelliptic. Then Clifford's Theorem (\cite[Thm. IV.5.4]{Hart}) implies that $W_2^1(X)$ is empty, that $\mathcal{W}^0_2(X)$ is therefore smooth, and that $C\mathcal{G}_2(X)$ is a line bundle over $\mathcal{W}^0_2(X)$. More generally, if $0 \leq d \leq 2g-2$, then $C\mathcal{G}_d(X)$ is smooth if and only if $d < g$ and $\mathcal{W}^0_d(X)$ is smooth. Recalling that $\mathcal{W}^0_d(X)$ is smooth if and only if $\mathcal{W}^1_d$ is empty, the smoothness of $C\mathcal{G}_d(X)$ depends on the existence of a $g^1_d$ on $X$. When $d \geq \frac{1}{2}g + 1$, every curve of genus $g$ has a $g^1_d$, implying that $C\mathcal{G}_d(X)$ can only be smooth when $d > 2g-2$ or $d < \frac{1}{2}g + 1$. When $d < \frac{1}{2}g + 1$, there always exists a curve of genus $g$ which does not possess a $g^1_d$. See \cite{KL}.
\end{exmp}

To resolve $C\mathcal{G}_d(X)$, consider the stack $\widetilde{C\mathcal{G}}_d(X)$, which fits into the following diagram.
\begin{diagram}
\widetilde{C\mathcal{G}}_d(X) & \rInto & C\mathcal{G}_d(X)
\underset{\mathcal{W}^0_d(X)}{\times}\mathcal{G}^0_d(X) & \rTo &
\mathcal{G}^0_d(X) \\
& \rdTo^u & \dTo && \dTo \\
&& C\mathcal{G}_d(X) & \rTo & \mathcal{W}^0_d(X)
\end{diagram}

$\widetilde{C\mathcal{G}}_d(X)$ is defined to be the closed substack of $C\mathcal{G}_d(X) \underset{\mathcal{W}^0_d(X)}{\times}\mathcal{G}^0_d(X)$ consisting of triples $(\lambda,s,\ell)$ such that $s \in \ell$.

\begin{prop}
$\widetilde{C\mathcal{G}}_d(X) \to C\mathcal{G}_d(X)$ is a
resolution of singularities.
\end{prop}
\begin{proof}
The map $u:\widetilde{C\mathcal{G}}_d(X) \to C\mathcal{G}_d(X)$ is
proper because $\mathcal{G}^0_d(X) \to \mathcal{W}^0_d(X)$ is
proper. To show that it is a birational equivalence, consider two
cases.
\begin{enumerate}
\item Suppose that $d < g$. Let $z:\mathcal{W}^0_d(X) \to
C\mathcal{G}_d(X)$ be the zero section, and consider the substack of
$C\mathcal{G}_d(X)$ defined as the image $z(\mathcal{W}^1_d(X))$.
Since $\mathcal{W}^1_d(X)$ is a closed substack of $W^0_d(X)$,
$z(\mathcal{W}^1_d(X))$ is a closed substack of $C\mathcal{G}_d(X)$.
Then $u$ is a birational equivalence because it is an isomorphism on
the complement of $z(\mathcal{W}^1_d(X))$.
\item Suppose that $g \leq d \leq 2g-2$. Then $u$ is a birational
equivalence because it is an isomorphism over the complement of the
image of the zero section $z(\mathcal{W}^0_d(X))$.
\end{enumerate}
$\widetilde{C\mathcal{G}}_d(X)$ is smooth because
$\mathcal{G}^0_d(X)$ is smooth and the projection map
$\widetilde{C\mathcal{G}}_d(X) \to \mathcal{G}^0_d(X)$ is smooth of
relative dimension $1$.
\end{proof}

\begin{rem}
Unfortunately, when $g < d \leq 2g-2$, $\widetilde{C\mathcal{G}}_d(X)$ is a somewhat unsatisfactory resolution of $C\mathcal{G}_d(X)$. This is because it is not an isomorphism over the smooth part of $C\mathcal{G}_d(X)$. Indeed, recall that there are smooth points of $C\mathcal{G}_d(X)$ lying in the image of the zero section, corresponding to those pairs $(\lambda,0)$ such that $H^1(X,\lambda) = 0$. However, $u^{-1}((\lambda,0)) \simeq \mathbb{P}H^0(X,\lambda)$, which is a projective space of dimension at least $d-g$. This does show, however, that $u$ is an isomorphism over the smooth part of $C\mathcal{G}_d(X)$ when $d=g$.
\end{rem}

For any $d \geq -\frac{1}{2}\deg(\L)$ such that $2d + \deg(\L) \leq
2g-2$, set
\begin{equation*}
\widetilde{\N}_d := \widehat{\N}_d \underset{C\mathcal{G}_{2d+\deg(\L)}(X)}{\times} \widetilde{C\mathcal{G}}_{2d+\deg(\L)}(X).
\end{equation*}
In other words, $\widetilde{\N}_d$ parameterizes the data of $(\lambda \subset E, \varphi,\ell)$ such that $(\lambda \subset E,\varphi) \in \widehat{\N}_d$, $\ell \in \mathbb{P}H^0(X,\lambda^{\otimes 2} \otimes \L)$, and $(\overline{\varphi}:\lambda^{-1} \to \lambda \otimes \L) \in \ell$.

Finally, define
\begin{equation*}
\widetilde{\N} := \underset{d=-\frac{1}{2}\deg(\L)}{\overset{g-1-\frac{\deg(\L)}{2}}{\bigcup}}\quad \widetilde{\N}_d \quad \cup \quad \underset{d > g-1-\frac{\deg(\L)}{2}}{\bigcup}\quad \widehat{\N}_d.
\end{equation*}

\begin{thm}
$\widetilde{N}$ is a resolution of singularities of $\N$.
\end{thm}

$\widetilde{\N}$ will be referred to as a \em global Springer
resolution \em of $\N$.

\begin{cor}
Each connected component of $\widetilde{\N}$ has dimension $\deg(\L) + g - 1$. $\N$ is therefore equidimensional of dimension $\deg(\L) + g - 1$.
\end{cor}
\begin{proof}
To compute the dimension of the connected component
$\widetilde{\N}_d$, it suffices to compute the relative dimension of
\begin{equation*}
r_d:Bun_{B,d}(X) \to Bun_{\mathbb{G}_m,d}(X).
\end{equation*}
Since the fiber $r_d^{-1}(\lambda)$ is the quotient stack
\begin{equation*}
\Ext^1(\lambda^{-1},\lambda) \big/ \Hom(\lambda^{-1},\lambda),
\end{equation*}
the relative dimension of $r_d$ is given by the negative Euler
characteristic
\begin{equation*}
-\chi(X,\lambda^{\otimes 2}) = -2d + (g-1).
\end{equation*}
Therefore,
\begin{equation*}
\dim(\widetilde{\N}_d) = \dim(C\mathcal{G}_d(X,\L)) + \dim(r_d) =
\deg(\L) + g - 1.
\end{equation*}
\end{proof}

\subsection{Twisting bundles of smaller degree and stable bundles}\label{ss:lower degree}

Recall the restrictions placed on the twisting bundle $\L$ in the previous sections. Besides requiring the degree of $\L$ to be even, there was also the restriction that
\begin{equation*}
\deg(\L) \geq 2g.
\end{equation*}
The purpose of this lower bound on the degree of $\L$ is to ensure that for every $SL_2$-bundle $E$, there exists a nonzero nilpotent twisted endomorphism $\varphi:E \to E \otimes \L$. In other words, the lower bound on $\deg(\L)$ makes it so that the image of the zero section does not form its own irreducible component of $\N$. To see how the particular lower bound on $\deg(\L)$ was chosen, consider the following simple lemmas.

\begin{lem}\label{nonzero nilpotent}
Given an $SL_2$-bundle $E$, there exists a nonzero nilpotent twisted
endomorphism $\varphi:E \to E \otimes \L$ if and only if $E$
possesses a line subbundle $\lambda$ such that
$h^0(X,\lambda^{\otimes 2} \otimes \L) \geq 1$.
\end{lem}
\begin{proof}
If there exists a nonzero nilpotent twisted endomorphism $\varphi$ of $E$, then we may take $\lambda = \ker(\varphi)$.

Conversely, suppose that $\lambda \subset E$ is a line bundle such that $h^0(X,\lambda^{\otimes 2} \otimes \L) \geq 1$. Then pick a nonzero section $s \in H^0(X,\lambda^{\otimes 2} \otimes \L)$, and consider $\varphi:E \to E \otimes \L$ defined as the composition
\begin{equation*}
E \to E/\lambda \simeq \lambda^{-1} \overset{s}{\to} \lambda \otimes
\L \hookrightarrow E \otimes \L.
\end{equation*}
Then $\varphi$ is a nonzero nilpotent twisted endomorphism of $E$.
\end{proof}

\begin{lem}\label{subsheaf}
If $\lambda$ is a line bundle on $X$ such that
\begin{equation*}
\deg(\lambda) \leq -g,
\end{equation*}
then, for all $SL_2$-bundles $E$ on $X$, $\lambda$ is a subsheaf of
$E$.
\end{lem}
\begin{proof}
It is equivalent to show that $H^0(X,E \otimes \lambda^{-1}) \neq
0$. By Riemann-Roch,
\begin{equation*}
\chi(E \otimes \lambda^{-1}) = 2(-\deg(\lambda)+1-g) > 0.
\end{equation*}
\end{proof}

The next proposition justifies the assumption that $\deg(\L) \geq 2g$.

\begin{prop}
If $\deg(\L) \geq 2g$, then any $SL_2$-bundle $E$ possesses a
nonzero nilpotent twisted endomorphism $\varphi:E \to E \otimes \L$.
\end{prop}
\begin{proof}
Let $\lambda$ be the inverse of a square root of $\L$. By Lemma
\ref{subsheaf}, $\lambda$ is a subsheaf of $E$. If $\lambda$ is
furthermore a subbundle of $E$, then $h^0(X,\lambda^{\otimes
2}\otimes \L) = 1$, in which case Lemma \ref{nonzero nilpotent}
finishes the proof.

Otherwise, the normalization $\widetilde{\lambda} \supset \lambda$
is a subbundle of $E$. Then $\lambda^{\otimes 2}\otimes \L$ is a
subsheaf of $\widetilde{\lambda}^{\otimes 2} \otimes \L$, which
implies that
\begin{equation*}
h^0(\widetilde{\lambda}^{\otimes 2}\otimes \L) \geq
h^0(\lambda^{\otimes 2}\otimes \L) = 1.
\end{equation*}
Then Lemma \ref{nonzero nilpotent} finishes the proof in this case
as well.
\end{proof}

Now suppose that $\deg(\L) \leq 2g-2$, while keeping the assumption
that $\deg(\L)$ is even. Let $Z$ denote the image of the zero
section
\begin{equation*}
z:\Bun_{SL_2}(X) \to \N.
\end{equation*}
$Z$ is an irreducible component of $\N$, which is smooth because $\Bun_{SL_2}(X)$ is smooth. Therefore, to resolve $\N$ it suffices to resolve the irreducible components apart from $Z$. However, it is clear that $\widetilde{\N}$ (resp., $\widehat{\N}$ when $g \leq 1$) still serves this purpose, regardless of the degree of $\L$.
\begin{thm}
If $\L$ is an even degree line bundle such that $\deg(\L) \leq
2g-2$, then the disjoint union $\widetilde{\N} \sqcup Z$ (resp.,
$\widehat{\N} \sqcup Z$) is a resolution of singularities of $\N$.
\end{thm}

Having now discussed $\N$ for any even degree line bundle $\L$, let us briefly turn our attention to stable Higgs bundles. We assume that $g \geq 2$ for the remainder of the section, as there are no stable $SL_2$-bundles when $g < 2$. The following definition first appeared in \cite{Hi}.

\begin{dfn}
For any vector bundle $E$, the \em slope \em of $E$ is defined to be
\begin{equation*}
\mu(E) := \frac{\deg(E)}{\rk(E)}.
\end{equation*}
A Higgs bundle $(E,\varphi)$ is said to be \em stable \em if, for any vector bundle $F$ which is a $\varphi$-invariant subsheaf of $E$,
\begin{equation*}
\mu(F) < \mu(E).
\end{equation*}
\end{dfn}

Let $\N^s \subset \N$ denote the substack of stable nilpotent Higgs bundles. Note that if $(E,\varphi)$ is a stable nilpotent Higgs bundle, then $F \subset E$ is a $\varphi$-invariant subsheaf if and only if $F \subset \ker(\varphi)$. Let $\widetilde{\N}{}^s_d$ denote the restriction of $\widetilde{\N}_d$ to $\N_d^s$, and set
\begin{equation*}
\widetilde{\N}{}^s =
\overset{-\frac{1}{2}\deg(\L)}{\underset{d=-1}{\bigcup}}\quad
\widetilde{\N}{}^s_d.
\end{equation*}

\begin{prop}\label{stable}
First suppose that $\deg(\L) \geq 2g$. Then $\widetilde{\N}{}^s$ is
a resolution of singularities of $\N^s$. If $0< \deg(\L) \leq 2g-2$,
then $\widetilde{\N}{}^s \sqcup Z$ is a resolution of singularities
of $\N^s$. Lastly, if $\deg(\L) \leq 0$, then $\N^s = Z$ is smooth.
\end{prop}

In particular, Proposition \ref{stable} provides a count of the finite number of irreducible components of $\N^s$.
\begin{cor}
\begin{enumerate}
\item If $\deg(\L) \geq 2g$, then $\N^s$ has
$\frac{1}{2}\deg(\L)$ irreducible components.
\item If $0< \deg(\L) \leq 2g-2$, then $\N^s$ has
$\frac{1}{2}\deg(\L)+1$ irreducible components.
\item If $\deg(\L) \leq 0$, then $\N^s$ has $1$ irreducible component.
\end{enumerate}
\end{cor}

\bibliographystyle{amsalpha}
\bibliography{GSR}

\end{document}